\documentclass[reqno,11pt]{amsart}

\usepackage{amsmath,amsfonts,amssymb,amsthm,epsfig}
\usepackage{mathtools}
\usepackage[english]{babel}
\usepackage{tikz}
\usepackage[colorlinks,linkcolor=red,citecolor=red]{hyperref}
\usetikzlibrary{arrows}
\usepackage{comment}

 \allowdisplaybreaks
\numberwithin{equation}{section}

\font\script=rsfs10 at 11pt
\def\eps{\varepsilon}
\def\H{{\mbox{\script H}\,\,}}
\def\Hz{$(\hbox{\bf H})$}
\def\L{{\mbox{\script L}\,\,}}
\def\E{\mathcal E}
\def\P{\mathcal P}
\def\R{\mathbb R}
\def\S{\mathbb S}
\def\N{\mathbb N}
\def\bal{\begin{aligned}}
\def\eal{\end{aligned}}
\def\proofof#1{\begin{proof}[Proof of #1]}
\def\Chi#1{\hbox{{\large $\chi$}{\Large $_{_{#1}}$}}}

\def\step#1#2{\par\noindent{\underline{\it Step~#1.}}\emph{ #2}\\}
\def\freccia#1{\xrightarrow[\ #1]{}}
\def\XXint#1#2#3{{\setbox0=\hbox{$#1{#2#3}{\int}$} \vcenter{\vspace{-1pt}\hbox{$#2#3$}}\kern-.5\wd0}}
\def\Xint#1{\mathchoice {\XXint\displaystyle\textstyle{#1}}{\XXint\textstyle\scriptstyle{#1}}{\XXint\scriptstyle\scriptscriptstyle{#1}}{\XXint\scriptscriptstyle\scriptscriptstyle{#1}}\!\int}
\def\XXiint#1#2#3{{\setbox0=\hbox{$#1{#2#3}{\iint}$} \vcenter{\vspace{-1pt}\hbox{$#2#3$}}\kern-0.5\wd0}}
\def\Xiint#1{\mathchoice {\XXiint\displaystyle\textstyle{#1}}{\XXiint\textstyle\scriptstyle{#1}}{\XXiint\scriptstyle\scriptscriptstyle{#1}}{\XXiint\scriptscriptstyle\scriptscriptstyle{#1}}\!\iint}
\def\intmed{\Xint{\hbox{---}}}
\def\iintmed{\Xiint{\hbox{-----}}}
\def\comp{\subset\subset}
\def\spt{{\rm spt}}

\newcommand{\res}{\mathop{\hbox{\vrule height 7pt width .5pt depth 0pt
\vrule height .5pt width 6pt depth 0pt}}\nolimits}

\pgfarrowsdeclare{<<<}{>>>}
{
\arrowsize=0.2pt
\advance\arrowsize by .5\pgflinewidth
\pgfarrowsleftextend{-4\arrowsize-.5\pgflinewidth}
\pgfarrowsrightextend{.5\pgflinewidth}
}
{
\arrowsize=0.2pt
\advance\arrowsize by .5\pgflinewidth
\pgfpathmoveto{\pgfpointorigin}
\pgfpathlineto{\pgfpoint{-1.2mm}{-.8mm}}
\pgfusepathqstroke
\pgfpathmoveto{\pgfpointorigin}
\pgfpathlineto{\pgfpoint{-1.2mm}{.8mm}}
\pgfusepathqstroke
}

\pgfarrowsdeclare{<<<<}{>>>>}
{
\arrowsize=0.2pt
\advance\arrowsize by .5\pgflinewidth
\pgfarrowsleftextend{-4\arrowsize-.5\pgflinewidth}
\pgfarrowsrightextend{.5\pgflinewidth}
}
{
\arrowsize=0.2pt
\advance\arrowsize by .5\pgflinewidth
\pgfpathmoveto{\pgfpointorigin}
\pgfpathlineto{\pgfpoint{-1.8mm}{-1.2mm}}
\pgfusepathqstroke
\pgfpathmoveto{\pgfpointorigin}
\pgfpathlineto{\pgfpoint{-1.8mm}{1.2mm}}
\pgfusepathqstroke
}

\newcounter{mt}
\def\maintheorem#1#2#3{\par \medskip \noindent {\bf Theorem~\mref{#1}}~(#2).~{\it #3}\par}
\def\mref#1{\Alph{#1}}
\def\maintheoremdeclaration#1{\stepcounter{mt}\newcounter{#1}\setcounter{#1}{\arabic{mt}}}

\maintheoremdeclaration{existence}
\maintheoremdeclaration{Linfty}
\maintheoremdeclaration{Unique}

\newtheorem{theorem}{Theorem}[section]
\newtheorem{lemma}[theorem]{Lemma}
\newtheorem{prop}[theorem]{Proposition}

\newtheorem{defin}[theorem]{Definition}
\newtheorem{remark}[theorem]{Remark}

\begin{document}

\title[Minimizing measures for non-local energies]{On the existence and boundedness of minimizing measures for a general form of non-local energies}

\author{D. Carazzato}
\author{A. Pratelli}

\begin{abstract}
In this paper we consider a very general form of a non-local energy in integral form, which covers most of the usual ones (for instance, the sum of a positive and a negative power). Instead of admitting only sets, or $L^\infty$ functions, as admissible objects, we define the energy for all the Radon measures. We prove the existence of optimal measures in a wide generality, and we show that in several cases the optimal measures are actually $L^\infty$ functions, providing an \emph{a priori} bound on their norm. We also derive a uniqueness result for minimizers.
\end{abstract}

\maketitle

\section{Introduction}

In recent years, much work has been dedicated to study energies with a competition between an attractive and a repulsive term. One of the main sources of interest has been the celebrated liquid drop model by Gamow, which amounts in minimizing the energy
\[
P(E) + \iint_{E\times E} |y-x|^{-\alpha}\,dy\,dx
\]
among sets $E$ with a given volume in $\R^N$, where $P(E)$ is the perimeter of the set $E$ and $\alpha\in (0,N)$. The original model was actually with $N=3$ and $\alpha=1$, the extension to a general dimension and power is then natural. There have been several generalizations of this problem. One of them has been to substitute the perimeter with a fractional perimeter (see~\cite{F2M3}); another one, as done by several authors, has been to ``relax'' the problem, considering not only sets of a given volume, but more in general positive functions with values in $[0,1]$ and with a fixed $L^1$ norm. Both generalizations together have led Frank and Lieb in~\cite{FL} to study the minimization of the problem
\begin{equation}\label{FrLi}
\iint_{\R^N\times\R^N} \left(|y-x|^\beta + |y-x|^{-\alpha}\right) f(x) f(y)\,dy\,dx
\end{equation}
with $\beta>0,\, 0<\alpha<N$, among functions $f:\R^N\to [0,1]$ with fixed $L^1$ norm. Another generalisation has been to consider a repulsive term given by the double integral of a more general function of $|y-x|$, not necessarily a power (see for instance~\cite{NP,CFP,Yao}). In particular, the last paper considers a further extension of the admissible objects, which are functions in $L^1\cap L^\infty$, with $L^\infty$ norm not necessarily bounded by $1$.\par

The goal of this paper is to start studying an even more general version of the question. Indeed, on one side we consider as admissible objects the positive measures with a given mass (this possibility was already mentioned before, for instance in~\cite{BCT,CDM}). And on the other side, we consider energies given by a quite general function $\bar g(y-x)$. That is, we aim to study the minimization of the quantity
\[
\E(\mu)=\iint_{\R^N\times\R^N} \bar g(y-x)\, d\mu(y)\, d\mu(x)
\]
among positive measures of given mass. Notice that the functional is well defined as soon as $\bar g$ is l.s.c., and that $\bar g$ must be assumed to be $L^1_{\rm loc}$, since otherwise the energy of any measure is infinite. The ``prototype'' that we have in mind, considering the questions studied in the literature and in particular the energy~(\ref{FrLi}), is
\begin{equation}\label{prototype}
\bar g(v)= |v|^{-\alpha} + |v|^\beta
\end{equation}
with $0<\alpha<N$ and $\beta>0$, but we are able to deal with a wide class of functions $\bar g$. We will prove under very weak assumptions that optimal measures exist and have bounded support. Moreover, under stronger but still quite general assumptions, we will prove that optimal measures are in fact bounded functions. In particular, in these cases, the problem among measures has the same solutions than the problem among $L^\infty$ functions; in Remark~\ref{ConjFrLi} we discuss the meaning of this fact in connection with what is observed in~\cite{FL}. Both the existence and the regularity were already considered in recent papers, and shown under more rigid assumptions on the class of admissible objects or on the function $g$, see in particular~\cite{BCT} and~\cite{CDM}. In the next sections we will briefly describe these results, to make a comparison with the present setting.\par

It is important to observe immediately that the energy is $2$-homogeneous, that is, for any constant $\lambda>0$ we have $\E(\lambda \mu)=\lambda^2 \E(\mu)$. As a consequence, we can restrict ourselves to consider the minimization problem among probability measures, since minimizers for any mass are then always the same measures, up to a multiplicative constant. We remark that this is a huge difference with respect to the problem considered only in the class of sets, or of the functions bounded by $1$. Indeed, in those cases the minimizers are known heavily to depend on the mass that one is considering. This particular feature will be useful in Remark~\ref{ConjFrLi}.\par

We remark also that, in some cases, it is known that optimal measures are unique (up to translations, of course), hence in particular radial; this was done for instance in~\cite{BCT,Lop}, we will present a slightly revised version of their result in Theorem~\mref{Unique}. Also motivated by this, we will consider the minimization problem in two different classes; the first one is the class $\P$ of the probability measures on $\R^N$, and the second one is the subclass $\P_{\rm rad}$ made by the radial ones. Of course the two problems are equivalent under the assumptions that guarantee that minimizers are radial, but the problem for radial densities can be interesting also more in general.\par

We state now our main results. The first one concerns the existence of optimal measures; it is much simpler to prove and it does not require any particular assumption, the function $\bar g$ is not even assumed to be radial. We remark that the same result, or a slight variant of it, has been already proved in~\cite{SST,CCP}, we put it in our paper just to keep the it self-contained, and since it requires less than one page of proof.

\maintheorem{existence}{Existence of optimal measures}{Let $\bar g:\R^N\to\R^+$ be a l.s.c. and $L^1_{\rm loc}$ function such that $\lim_{|x|\to +\infty} \bar g(x)=+\infty$. Then, there exists a minimizer of the energy $\E$ both in the class $\P$ and in the class $\P_{\rm rad}$. Moreover, the support of any optimal measure is contained in a ball of radius $R$, where $R$ only depends on $\bar g$. More precisely, $R$ needs only to be big enough so that, for every $|v|>R/4$, the quantity $\bar g(v)$ is larger than $24$ times the energy of a ball of unit volume.}

Our second result concerns the $L^\infty$ property for optimal measures. It is much harder to obtain, and we can prove it in a quite wide context, but not as general as for Theorem~\mref{existence}. Also motivated by the ``standard case''~(\ref{prototype}), we will consider the following assumption (as usually done, we denote by $B(x,r)$ the ball centered at $x$ and with radius $r$, and write $B_r=B(0,r)$).
\begin{itemize}
\item[\Hz] The function $\bar g:\R^N\to\R^+$ is radial, $L^1_{\rm loc}$, and its restriction to $\R^N\setminus\{0\}$ is ${\rm C}^2$. In addition, calling $\bar g(v)=g(|v|)$, there is a small radius $r>0$ such that $\bar g$ is subharmonic in $B_r\setminus \{0\}$, $g$ and $g'$ are respectively decreasing and increasing in $(0,r)$, and $g(0)=\lim_{t\searrow 0} g(t)$.
\end{itemize}

We underline that the above assumption is extremely weak, and broadly covered by most of the examples already studied. Actually, a typical feature of any attraction-repulsion model is that the interaction energy at distance $t$ explodes for $t\searrow 0$, it is first decreasing and then increasing, and explodes again for $t\to\infty$. So all the requests of assumption \Hz, except the subharmonicity in $B_r\setminus\{0\}$, are satisfied by basically any interesting radial model. Concerning the subharmonicity, this is also true for many but not all models. In particular, in the prototype case~(\ref{prototype}), the function $\bar g$ satisfies \Hz\,whenever $\alpha\geq N-2$ (this is consistent with the powers considered in~\cite{CDM}). An important variant of~(\ref{prototype}) is when a power is replaced by a logarithm. In particular, when $N\geq 2$ then the function $g(t) = \ln(t) + t^{-\alpha}$ satisfies \Hz\,for any $\alpha\geq N-2$, while for $N\leq 2$ the function $g(t)=t^\beta - \ln(t)$ satisfies \Hz\,for any $\beta>0$.\par

Notice that assumption \Hz\,does not require that $g$ explodes at $0$, even though this is true in all the usual cases. Actually, if $g(0)$ is finite then several technical points in the proofs are much simpler; of course, Dirac masses have finite energy if and only if $g(0)$ is finite.\par

Our main result is then the following.
\maintheorem{Linfty}{$L^\infty$ bound for optimal measures}{If $\bar g:\R^N\to\R^+$ satisfies assumption~\Hz\,and $\lim_{t\to\infty} g(t)=+\infty$, there exists a constant $M=M(N,g)$ such that the $L^\infty$ bound $\|\mu\|_{L^\infty}\leq M$ is true in the following cases:
\begin{enumerate}
\item for any minimizer $\mu$, either in $\P$ or in $\P_{\rm rad}$, if the support of $\mu$ is convex and
\begin{equation}\label{addass2}
\limsup_{t\searrow 0} |g'(t)| t^N>0\,;
\end{equation}
\item for any minimizer $\mu$ in $\P_{\rm rad}$ (and also in $\P$ if $N=1$), if $\bar g$ is subharmonic in $\R^N\setminus\{0\}$, strictly subharmonic in some $B_r\setminus\{0\}$, and~(\ref{addass2}) holds;
\item for at least a minimizer $\mu$ in the class $\P_{\rm rad}$ (and also in $\P$ if $N=1$) if $\bar g$ is subharmonic in $\R^N\setminus\{0\}$.
\end{enumerate}
}

It is important to notice that, in the above claim, there is a big difference between the first two cases and the last one. More precisely, in the first two cases we prove that \emph{every} minimizer is in $L^\infty$, while in the third case we only prove the existence of \emph{some} bounded minimizer. The reason why this happens is that to obtain the boundedness of an optimal measure we need the convexity of the support, and this is explicitely asked in the first case; instead, in the other cases, part of the proof is to obtain this convexity, and to get it we need the ``global'' subharmonicity of $g$ (by global we mean ``in the whole $\R^N\setminus\{0\}$''). This global subharmonicity is not needed in case~(1) above since the convexity is guaranteed by assumption. In addition, the global subharmonicity can give the convexity of the support of \emph{every} optimal measure only if it is strict in some small punctured ball $B_r\setminus\{0\}$ and~(\ref{addass2}) holds, and this is why also case~(2) works for every optimal measure. In case~(3) we work by approximation, and this provides some $L^\infty$ minimizer, obtained as weak* limit of $L^\infty$ minimizers of the approximating problems, but not necessarily all the minimizers can be obtained in this way. This is not just a technical failure; indeed, we can easily notice that the $L^\infty$ property for every minimizer is in general false under the assumptions of case~(3) of the above Theorem. To notice that, it is enough to consider the map $g:(0,+\infty)\to\R^+$ given by $g(t)=1$ for $0\leq t \leq 1$ and $g(t)=(t-1)^3$ for $t\geq 1$. It is immediately seen that the corresponding map $\bar g$ satisfies \Hz\,with $r=1$ and is globally subharmonic, hence the above result guarantees the existence of some $L^\infty$ minimizer. But in fact, it is obvious that a measure is a minimizer if and only if its support has diameter less than $1$; therefore, there are $L^\infty$ minimizers, but also singular ones, for instance any Dirac mass.

\begin{remark}[The energies of the form~(\ref{prototype}) covered by Theorem~\mref{Linfty}]
With an elementary calculation, we can determine for which energies of the form~(\ref{prototype}) Theorem~\mref{Linfty} can be applied. In particular, case~(1) can be applied whenever $\alpha\geq N-1$, case~(2) whenever $\alpha\geq N-1$ and $\beta\geq 2-N$, and case~(3) whenever $\alpha\geq N-2$ and $\beta\geq 2-N$ (so in particular in all cases if $N=2$).
\end{remark}

Our last result concerns the uniqueness of the optimal measure --of course, up to a translation. As we said above, this uniqueness has been already established in some cases; more precisely, in~\cite{BCT} it is shown that it is true if $\bar g$ is of the form~(\ref{prototype}) with any $0<\alpha<N$ and $\beta=2$, while in~\cite{Lop} this is generalised to cover again any $0<\alpha<N$, and any $2\leq \beta\leq 4$. The uniqueness up to translation of minimizers is of course of great importance. Indeed, on one side it eliminates any difference between ``any minimizer'' and ``some minimizer'' in the above Theorem~\mref{Linfty}; and on the other side, it eliminates any difference between minimizing in $\P$ and in $\P_{\rm rad}$. Essentially putting together the approach of~\cite{Lop} and the properties of the positive definite functions (which will be formally defined in Section~\ref{sect:unique}) we obtain then the following result.

\maintheorem{Unique}{Uniqueness and radiality of optimal measures}{Let $\bar g(x)=\bar h(x)+|x|^\beta$ be a function satisfying \Hz, with $2\leq \beta\leq 4$. If $\bar h$ is strongly positive definite, there is some minimal measure $\bar\mu\in\P$ which belongs to $\P_{\rm rad}$, and if $\bar h$ is strictly strongly positive definite then $\bar\mu$ is the unique minimal measure up to translations. If $\bar h$ is subharmonic in $\R^N\setminus\{0\}$ and radially decreasing, then it is also strongly positive definite, so in particular there is some minimal measure $\bar\mu\in\P$ which belongs to $\P_{\rm rad}$. If in addition $\bar h$ is strictly subharmonic in some $B_r\setminus \{0\}$, then there is a radius $R_1$ such that the support of every minimal measure in $\P$ is a ball of radius $R_1$.}

We point out that Carrillo and Shu show in \cite{CS} some uniqueness and symmetry results for local minimizers of $\E$ with general conditions on $\bar g$ (and possibly requiring some regularity for the measures themselves or for their potential). It is worth to highlight that their uniqueness/symmetry results, as well as ours, heavily rely on the convexity of $\E$. It is not clear whether one can expect similar conclusions to hold in some other cases without convexity of the energy.\par\medskip

We conclude this introduction by presenting the plan of the paper, which is very symple. First of all, in Section~\ref{sec:exist}, we give the proof of Theorem~\mref{existence}. Then, Section~\ref{sec:bound} is devoted to show Theorem~\mref{Linfty}; to do so, we will first present the potential and study its main properties, in Section~\ref{ssect:prel}; then we study the convexity of the support of optimal measures in Section~\ref{ssect:convex}; then we present our main ``geometric'' estimates in Section~\ref{ssect:main}; and finally we put everything to get the proof of Theorem~\mref{Linfty} in Section~\ref{ssect:pf}. Afterwards, Section~\ref{sect:unique} is devoted to show Theorem~\mref{Unique}.

\section{The proof of Theorem~\mref{existence}\label{sec:exist}}

This section only contains the proof of Theorem~\mref{existence}, which only requires a rather standard argument.

\proofof{Theorem~\mref{existence}}
We can assume without loss of generality that $\bar g$ is symmetric, since the energy does not change if we replace it by $v\mapsto \big(\bar g(v)+\bar g(-v)\big)/2$. Since $\bar g\in L^1_{\rm loc}(\R^N;\R^+)$, then we have $I\leq I' <+\infty$, having set
\begin{align*}
I:= \inf \big\{ \E(\mu):\, \mu\in \P\big\}\,, && I':= \inf \big\{ \E(\mu):\, \mu\in \P_{\rm rad}\big\}\,.
\end{align*}
Let us call for brevity $C=24I'$, and let $R>0$ be such that $\bar g(v) > C$ for every $v\in\R^N,\, |v|>R/4$. Let us now take a measure $\mu$, either in $\P$ or in $\P_{\rm rad}$, such that $\E(\mu)<2I'$. We claim that there exists some $\bar x\in\R^N$ such that $\mu\big(B(\bar x,R/4)\big)>1/2$. Indeed, otherwise we have
\[
\E(\mu)= \iint_{\R^N\times\R^N} \bar g(y-x)\,d\mu(y)\,d\mu(x)
\geq \int_{\R^N} \int_{\R^N\setminus B(x,R/4)} \bar g(y-x)\,d\mu(y)\,d\mu(x)
>\frac C2> \E(\mu)\,,
\]
which is absurd. Then, the existence of $\bar x\in\R^N$ so that $\mu\big(B(\bar x,R/4)\big)>1/2$ follows. We can reduce ourselves to assume that
\[
\mu\big( B_{R/2})\big) > \frac 12\,.
\]
Indeed, in the general case when $\mu\in\P$ it is harmless to assume that $\bar x\equiv 0$, up to a translation, so there is even no need of passing from $R/4$ to $R/2$. Instead, in the radial case --where a translation is not possible-- the above estimate is clearly true if $|\bar x|\leq R/4$. And in turn, we can exclude that $|\bar x|>R/4$, because if this happens then the balls $B(\bar x,R/4)$ and $B(-\bar x,R/4)$ are disjoint, and since $\mu$ is radial we obtain $\mu(\R^N)\geq \mu(B(\bar x,R/4))+\mu(B(-\bar x,R/4)) =2\mu(B(\bar x,R/4))>1$, which is absurd.\par

Let us now call $\eta = \mu\big(\R^N\setminus B_R\big) \in [0,1/2]$, and let $\mu^-$ be the restriction of $\mu$ to $B_R$, that is a measure with mass $1-\eta$. Then, we have
\[\begin{split}
\E(\mu) &= \iint_{\R^N\times\R^N} \bar g(y-x)\,d\mu(y)\,d\mu(x)
\geq \E(\mu^-) + 2 \int_{B_{R/2}} \int_{\R^N\setminus B_R} \bar g(y-x)\,d\mu(y)\,d\mu(x)\\
&\geq \E(\mu^-) + 2 C \mu\big(B_{R/2}\big) \mu\big(\R^N\setminus B_R\big)
\geq \E(\mu^-) + C\eta\,.
\end{split}\]
Keeping in mind that $\E$ is $2$-homogeneous, that $C=24I'>12 \E(\mu)$, and that $(1-\eta)^{-2} \leq 1+6\eta$ since $0\leq \eta\leq 1/2$, we can estimate
\[
\E\big((1-\eta)^{-1}\mu^-\big) = (1-\eta)^{-2} \E(\mu^-) \leq (1+6\eta) \E(\mu^-)
\leq (1+6\eta) \big(\E(\mu) - C\eta\big)
\leq \E(\mu) - \frac C2\, \eta\,.
\]
Therefore, the measure $(1-\eta)^{-1}\mu$, which is a probability measure concentrated in $B_R$, and which is radial if so is $\mu$, has energy lower than $\mu$, and actually strictly lower unless $\mu$ itself is concentrated in $B_R$.\par

Summarizing, from any minimizing sequence for the energy (either in $\P$ or in $\P_{\rm rad}$) we can construct another minimizing sequence, which is done by measures concentrated in the ball $B_R$. By lower semicontinuity of the energy, any weak limit of this latter minimizing sequence is a minimizer (observe that a weak limit of radial measures is clearly still radial). This gives the required existence of minimizer of the energy both in $\P$ and in $\P_{\rm rad}$. Moreover, by the above calculation we obtain that \emph{every} minimizer is concentrated in a ball of radius $R$.
\end{proof}

\section{The proof of Theorem~\mref{Linfty}\label{sec:bound}}


This is the main section of the paper, and it is devoted to show Theorem~\mref{Linfty}. In particular, in Section~\ref{ssect:prel} we present the definition and the main properties of the potential; in Section~\ref{ssect:convex} we prove the convexity of the support of optimal measures in some cases; in Section~\ref{ssect:main} we give the main geometric estimates that we need for the proof of the Theorem; and in Section~\ref{ssect:pf} we present the proof.

\subsection{The potential and its main property\label{ssect:prel}}

The aim of this first section is to show that the potential constantly attains its minimum on the support of any minimising measure. Let us be more precise. First of all, we present the (standard) definition of potential. Notice that, in order to define the potential, we assume the function $\bar g$ to be symmetric (but not necessarily radial). However, as already noticed in the proof of Theorem~\mref{existence}, this assumption can always be done without loss of generality, since the problem with the function $\bar g$ is completely equivalent to the problem with the function $v\mapsto \big(\bar g(v)+\bar g(-v)\big)/2$.

\begin{defin}[Potential]
Given a l.s.c., symmetric and $L^1_{\rm loc}$ function $\bar g:\R^N\to\R^+$, for any positive measure $\mu$ we call \emph{potential associated to $\mu$} the function $\psi_\mu:\R^N\to\R^+$ defined as
\[
\psi_\mu(x) = \int_{\R^N} \bar g(y-x) \,d\mu(y)\,.
\]
Similarly, for any function $f\in L^1(\R^N)$, either positive or bounded and compactly supported, the \emph{potential associated to $f$} is the function $\psi_f:\R^N\to\R$ given by
\[
\psi_f(x) = \int_{\R^N} \bar g(y-x) f(y)\,dy\,.
\]
\end{defin}

An immediate but crucial property of the potential is that
\begin{equation}\label{proppot}
\E(\mu) = \int_{\R^N} \psi_\mu(x)\, d\mu(x)\,.
\end{equation}
As a consequence, it is easy to guess that, whenever $\mu$ is an optimal measure, the potential attains its minimum over the support of $\mu$. A similar result has been already proved under different assumptions in many earlier papers, for instance~\cite{BCT,CDM}. We now prove this fact under very weak assumptions, which do not even guarantee the existence of optimal measures --of course, if $\bar g$ is so that there are no minimizers of the energy, then the result is emptily true.
\begin{prop}\label{lowerlevel}
Let $\bar g:\R^N\to\R^+$ be a l.s.c., symmetric and $L^1_{\rm loc}$ function. Let $\mu$ be a minimizer of the energy, either in $\P$ or in $\P_{\rm rad}$ (in this latter case we also assume $\bar g$ to be radial). Then we have
\begin{align}\label{cost>}
\psi_\mu(x) = \E(\mu) \quad \hbox{for $\mu$-a.e. $x\in\spt(\mu)$}\,, && \psi_\mu(x) \geq \E(\mu) \quad \hbox{for $\L^N$-a.e. $x\in\R^N$}\,.
\end{align}
\end{prop}
\begin{proof}
We start by showing that $\psi_\mu$ is constant $\mu$-a.e. in the support of $\mu$. The fact that this constant is exactly $\E(\mu)$ will then be an obvious consequence of~(\ref{proppot}). If the claim is false, then there are two constants $\lambda_1<\lambda_2$ and two measures $\mu',\, \mu''\leq \mu$ with $\|\mu'\|=\|\mu''\|>0$, radial if so are $\mu$ and $\bar g$, and such that
\begin{align*}
\psi_\mu(x) \leq \lambda_1 \quad \hbox{for $\mu$-a.e. $x\in\spt(\mu')$}\,, && \psi_\mu(x) \geq \lambda_2 \quad \hbox{for $\mu$-a.e. $x\in\spt(\mu'')$}\,.
\end{align*}
For any $0<\eps<1$, the measure $\mu_\eps = \mu + \eps(\mu'-\mu'')$ is still a positive, probability measure, and it is radial if so is $\mu$. An easy calculation gives us that
\[
\E(\mu_\eps)-\E(\mu)= 2 \eps \iint_{\R^N} \psi_\mu(x) d(\mu'-\mu'')(x)+\eps^2 \E(\mu'-\mu'')
\leq 2 \eps \|\mu'\| (\lambda_1-\lambda_2)+\eps^2 \E(\mu'-\mu'')\,,
\]
and then we derive that $\E(\mu_\eps)<\E(\mu)$ for $\eps\ll 1$, contradicting the minimality of $\mu$. The first property in~(\ref{cost>}) is then established.\par

Concerning the second one, let us assume that it is false. Then, there exists some $\lambda< \E(\mu)$ and some bounded Borel set $E\subseteq\R^N$, with strictly positive Lebesgue measure, such that
\[
\psi_\mu(x) \leq \lambda \quad\hbox{for $\L^N$-a.e. $x\in E$}\,.
\]
The set $E$ can be taken radially symmetric if $\mu$ and $\bar g$ are radial. Notice that, by the first property in~(\ref{cost>}), $\mu(E)=0$. This time, for $0<\eps<|E|^{-1}$ we set
\[
\mu_\eps= (1-\eps |E|)\mu + \eps \L^N\res E= \mu +\eps \big( \L^N\res E - |E|\mu\big)\,,
\]
which is again a positive probability measure, radial if so are $\mu$ and $\bar g$ (and then $E$). We have
\[\begin{split}
\E(\mu_\eps)-\E(\mu) &= 2\eps \int_{\R^N} \psi_\mu(x) \, d\big(\L^N\res E(x) - |E|\mu\big)(x)+\eps^2 \E\big( \L^N\res E - |E|\mu\big)\\
&\leq 2\eps (\lambda -\E(\mu))|E|+\eps^2 \E\big( \L^N\res E - |E|\mu\big)\,.
\end{split}\]
Since $E$ is bounded and $\bar g\in L^1_{\rm loc}$, and since $\mu$ has finite energy and is compactly supported by Theorem~\mref{existence}, we derive that $\E(\L^N\res E-|E|\mu)<+\infty$. By the fact that $\lambda<\E(\mu)$, we deduce that $\E(\mu_\eps)<\E(\mu)$ for $\eps\ll 1$, contradicting the minimality of $\mu$. Also the second property in~(\ref{cost>}) is then obtained.
\end{proof}

\begin{remark}\label{forevery}
Keeping in mind that $\psi_\mu$ is l.s.c. on $\spt\mu$ and continuous on $\R^N\setminus\spt\mu$, from~(\ref{cost>}) we actually deduce
\begin{align*}
\psi_\mu(x) \leq \E(\mu) \quad \hbox{for every $x\in\spt(\mu)$}\,, && \psi_\mu(x) \geq \E(\mu) \quad \hbox{for every $x\notin\spt(\mu)$}\,.
\end{align*}
\end{remark}

\subsection{The convexity of the support of optimal measures\label{ssect:convex}}

In this section we show that the support of an optimal measure is convex in some cases. As in Section~\ref{ssect:prel}, our assumptions are not strong enough to guarantee the existence of optimal measures, hence what we prove is that every minimizing measure, if any, has convex support. This kind of result was already present in the proof of \cite[Theorem 4.1]{CS}, but they made some different hypotheses on the kernel $\bar g$ and some a-priori regularity assumption on the potential generated by an optimal measure. On the other hand, they work with local minimizers with respect to the $\infty$-Wasserstein distance, while we are interested only in the global minimizers of $\E$. We start with the $1$-dimensional case.

\begin{prop}\label{convexN=1}
Let $\bar g:\R\to \R^+$ be a l.s.c., symmetric and $L^1_{\rm loc}$ function, whose restriction to $(0,+\infty)$ is convex, and strictly convex in a right neighborhood of $0$. Let $\mu$ be a measure which minimizes the energy either in $\P$ or in $\P_{\rm rad}$. Then the support of $\mu$ is a closed segment.
\end{prop}
\begin{proof}
Let us assume that $\mu$ is a minimal measure, either in $\P$ or in $\P_{\rm rad}$, and that its support is not a segment. As a consequence, there is an open segment $(a,b)\subseteq\R$ such that $\spt(\mu)$ does not intersect $(a,b)$, but it contains both $\{a\}$ and $\{b\}$.\par

By construction, the function $\psi_\mu$ is convex in the interval $(a,b)$. Moreover, it is \emph{strictly} convex in $(a,a+\eps)$ and in $(b-\eps,b)$ for some $\eps>0$, much smaller than $b-a$. By Proposition~\ref{lowerlevel}, we deduce that $\psi_\mu\geq \E(\mu)$ in the whole open segment $(a,b)$, and that, up to possibly decrease the value of $\eps>0$, the inequality is strict in $(a,a+\eps)\cup (b-\eps,b)$. Consequently, and again up to further decrease $\eps$, the function $\psi_\mu$ is either strictly decreasing in $(a,a+\eps)$, or strictly increasing in $(b-\eps,b)$, or both. By symmetry, we assume without loss of generality that $\psi_\mu$ is strictly decreasing in $(a,a+\eps)$. Let us now notice that
\begin{equation}\label{ofas}
\psi_\mu(a) + \psi_\mu(a+\eps) - 2\psi_\mu(a+\eps/2) = \int_{x\in\R} \bar g(a-y) + \bar g(a+\eps-y) - 2 \bar g(a+\eps/2 - y)\,d\mu(y)\,.
\end{equation}
Since $\bar g$ is convex in $(0,+\infty)$ and symmetric, and since $\eps < b-a$, for $\mu$-a.e. $y$ we have that
\[
\bar g(a-y) + \bar g(a+\eps-y) - 2 \bar g(a+\eps/2 - y)\geq 0\,.
\]
Inserting this estimate in~(\ref{ofas}) we deduce that
\[
\psi_\mu(a) \geq 2 \psi_\mu(a+\eps/2) - \psi_\mu(a+\eps) > \psi_\mu(a+\eps/2)\,,
\]
where we have also used that $\psi_\mu$ is strictly decreasing in $(a,a+\eps)$. And finally, this is absurd since $\psi_\mu(a+\eps/2)>\E(\mu)$, as already noticed, while $\psi_\mu(a)\leq \E(\mu)$ by Proposition~\ref{lowerlevel} and since $\psi_\mu$ is l.s.c. by construction.
\end{proof}

The idea of the proof in the general case when $N\geq 2$ is similar, one only needs more care in the construction. A geometrical property that we are going to use is the following one.
\begin{lemma}\label{elgele}
For any $N\geq 2$, there exists a geometrical constant $C_N>1$ such that, if $\delta,\, \eta,\, d,\, r$ are four positive numbers such that
\begin{align*}
\eta > C_N \delta\,, && d> C_N\eta,\, && r >C_N d\,,
\end{align*}
then one has
\begin{equation}\label{claimelgele}
\frac{\H^{N-1}\Big(\Big\{ x\in\partial B(0,r):\,\big| x - (r-\eta){\rm e}_1 \big|\in (d,d+\delta)\Big\}\Big)}{(N-1)\omega_{N-1} d^{N-2} \delta} \in\bigg[ \frac 12\, ,\, 2 \bigg]\,.
\end{equation}
\end{lemma}
\begin{proof}
This is an elementary geometrical property, easy to establish with the aid of Figure~\ref{Fig:geom}. Let us consider four constants $\delta<\eta<d<r$, each quite smaller than the following one. Let us call $P= (r-\eta){\rm e}_1$ as in the figure. The points of $\partial B(0,r)$ having distance exactly $d$ from $P$ are the intersection between the spheres $\partial B(0,r)$ and $\partial B(P,d)$, hence they are a $(N-2)$-dimensional sphere contained in a hyperplane orthogonal to the direction ${\rm e}_1$. 
\begin{figure}[thbp]
\begin{tikzpicture}[>=>>>>]
\draw (10,0) arc(0:15:10);
\draw (10,0) arc(0:-15:10);
\fill(0,0) circle (2pt);
\draw (0,0) node[anchor=north east] {$O$};
\fill(9,0) circle (2pt);
\draw (9,0) node[anchor=north east] {$P$};
\draw[<->] (0.097,-0.022) -- (9.69,-2.24);
\draw (4.87,-1.12) node[anchor=north] {$r$};
\draw[<<<->>>] (9.075,0) -- (9.975,0);
\draw (9.5,0) node[anchor=north] {$\eta$};
\draw[<<<->>>] (9.047,0.087) -- (9.80,1.72);
\draw (9.45,0.75) node[anchor=south east] {$d$};
\draw[line width=2pt] (9.85,1.74) arc(10:12:10);
\draw[line width=2pt] (9.85,-1.74) arc(-10:-12:10);
\draw[<<<->>>] (10.2,1.74) -- (10.2,2.1);
\draw (10.3,1.92) node[anchor=west] {$\approx \delta$};
\draw[<<<->>>] (8.65,0) -- (8.65,1.74);
\draw (8.7,0.87) node[anchor=east] {$\rho_0$};
\draw[<<<->>>] (8.1,0) -- (8.1,2.1);
\draw (8.15,1.05) node[anchor=east] {$\rho_\delta$};
\draw[dotted] (8.1,2.1)-- (9.78,2.1);
\draw[dotted] (8.65,1.74)-- (9.85,1.74);
\end{tikzpicture}
 \caption{The situation in Lemma~\ref{elgele}.}\label{Fig:geom}
\end{figure}
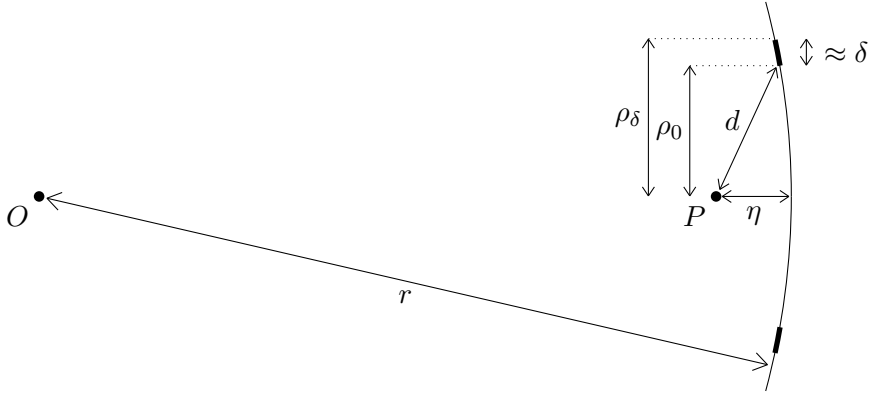
The radius of this sphere, call it $\rho_0$, is smaller than $d$, but the ratio between $\rho_0$ and $d$ becomes arbitrarily close to $1$ if the ratii $d/\eta$ and $r/d$ are both large enough. In the very same way, for any $0\leq t \leq \delta$, the points of $\partial B(0,r)$ having distance exactly $d+t$ from $P$ are a $(N-2)$-dimensional sphere, with radius $\rho_t$ very close to $d+t$. Moreover, $\rho_\delta-\rho_0\approx \delta$, that is, the ratio between $\rho_\delta-\rho_0$ and $\delta$ is arbitrarily close to $1$ as soon as $\eta/\delta,\, d/\eta,\, r/d$ are large enough. In addition, the centers of all these spheres are all on the line $\R {\rm e}_1$, and they are almost coincident with respect to $\delta$. More formally, if we call $C_t$ the center of the sphere corresponding to any $0\leq t \leq \delta$, we have that the ratio $|C_t - C_s|/|t-s|$ is arbitrarily close to $0$ as soon as $\eta/\delta,\, d/\eta,\, r/d$ are large enough.\par

Summarizing, the $\H^{N-1}$-measure of the union of these spheres is arbitrarily close to the measure of a $(N-1)$-dimensional annulus contained between two concentric spheres of radii $d$ and $d+\delta$, which in turn is arbitrarily close to $(N-1)\omega_{N-1} d^{N-2} \delta$ if $d/\delta$ is large enough. This completes the proof (in particular, instead of $1/2$ and $2$ we could have used $a$ and $1/a$ for any $a<1$).
\end{proof}

\begin{prop}\label{convexN>1}
Let $\bar g:\R^N\setminus\{0\}\to\R^+$ be a radial, ${\rm C}^2$, l.s.c., $L^1_{\rm loc}$ function, which is subharmonic in $\R^N\setminus \{0\}$ and strictly subharmonic in $B_r\setminus\{0\}$ for some $r>0$. Then, the support of any measure which minimizes the energy in $\P_{\rm rad}$ is a closed ball.
\end{prop}
\begin{proof}
Let $\mu$ be a measure minimizing the energy in $\P_{\rm rad}$. Since both $\bar g$ and $\mu$ are radial, then so is also the potential $\psi_\mu$. Let us define for brevity $f:\R^+\to\R^+$ the function such that $\psi_\mu(x)=f(|x|)$. Let us assume that the support of $\mu$ is not a closed ball, and let us look for a contradiction. Among all the open bounded intervals $I$ in $(0,+\infty)$ such that the annulus $\{x\in\R^N:\, |x|\in I\}$ does not intersect $\spt(\mu)$, there is at least one, say $(a,b)$, which is maximal with respect to the inclusion.\par

Notice that $\psi_\mu$ is a subharmonic radial function on $\R^N\setminus \spt(\mu)$, hence in particular we have
\begin{equation}\label{tttsub}
f''(t) + \frac{N-1}t\, f'(t)\geq 0 \qquad \hbox{in $(a,b)$}\,.
\end{equation}
We subdivide our proof in few steps. In the first one, we show that $f$ cannot be flat close to both $a$ and $b$, and in the following steps we reach a contradiction in each of the possible cases.

\step{I}{There is some $\eps>0$ such that either $f' > \eps$ in $(b-\eps,b)$, or $f'<-\eps$ in $(a,a+\eps)$.}
First of all, we want to show the existence of a small $\eps>0$ such that either $f'>\eps$ in $(b-\eps,b)$ or $f'<-\eps$ in $(a,a+\eps)$. Since by construction $b>0$, it is clear by~(\ref{tttsub}) that, if $f'(t)> 0$ for some $t<b$ close enough to $b$, then the value of $f'$ is at least $f'(t)/2$ in the whole interval $(t,b)$, and then we have already concluded this step. On the other hand, let us assume that $f'(t)\leq 0$ for every $t<b$ close enough to $b$. Since by construction the sphere $\partial B(0,b)$ belongs to $\spt(\mu)$, then $\psi_\mu$ is \emph{strictly} subharmonic in the annulus $\{x\in\R^N,\, b-\eta< |x|<b\}$ for $\eta\ll 1$, and this means that $f'(t)<0$ for some $t<b$ close to $b$. But then, (\ref{tttsub}) implies that $f'(s)<f'(t)$ for every $a<s<t$, and then the step is concluded.

\step{II}{Proof if $f'<-\eps$ in $(a,a+\eps)$ and $a=0$.}
We first assume that $f'<-\eps$ in $(a,a+\eps)$ for some small $\eps$. As a consequence, we can deduce that the sphere $\partial B(0,a)$ belongs to $\spt(\mu)$, but only if $a>0$. Let us instead suppose in this step that $a=0$. The fact that $f'<-\eps$ in a right neighborhood of $0$ implies that $\psi_\mu$ is not regular at the origin, having a cusp point. However, by construction $\psi_\mu$ is regular in $\R^N\setminus \spt(\mu)$, and then we deduce that the origin belongs to $\spt(\mu)$. Since the annulus $\{x\in\R^N,\, 0<|x|<b\}$ does not intersect $\spt(\mu)$, this means that the origin is an isolated point of $\spt(\mu)$. But since $\mu$ minimizes the energy, so in particular $\E(\mu)<+\infty$, the presence of an isolated point is only possible if $\bar g(0)<+\infty$. And finally, if $\bar g(0)$ is finite, then $\psi_\mu$ is clearly continuous, and we find a contradiction because we should have $\lim_{t\searrow 0} f(t) > \E(\mu)$ since $f$ is strictly decreasing in a right neighborhood of $0$ and $f\geq \E(\mu)$ a.e. in $(a,b)$ by Proposition~\ref{lowerlevel}. And again by Proposition~\ref{lowerlevel}, we have $f(0)=\E(\mu)$, obtaining the searched contradiction.

\step{III}{Proof if $f'<-\eps$ in $(a,a+\eps)$ and $a>0$.}
We now assume again that $f'<-\eps$ in $(a,a+\eps)$, but $a>0$. As already noticed before, this implies that $\partial B(0,a)\subseteq \spt(\mu)$, and by Proposition~\ref{lowerlevel} and the lower semicontinuity of $\psi_\mu$ we deduce that $f(a)\leq \E(\mu)$. On the other hand, $\lim_{t\searrow a} f(t)> \E(\mu)$, and then $f$ has a jump point at $a$, with $f(a) < \lim_{t\searrow a} f(t)$. We can easily show that this is impossible. Indeed, the discontinuity of $\psi_\mu$ implies that $\bar g$ is not bounded around the origin, and since $\bar g$ is subharmonic this implies that $\bar g$ is a radial, decreasing function in a neighborhood of the origin. In other words, calling $g:\R\to\R^+$ the function such that $\bar g(x)=g(|x|)$, up to possibly decrease the value of $\eps$ we have that $g$ is strictly decreasing in $(0,2\eps)$ and $\lim_{t\searrow 0} g(t)=+\infty$.\par

Let us now call $\bar x=a {\rm e}_1$, and $w=(a+\delta){\rm e}_1$ for some $\delta\ll \eps$. Since as noticed before $\lim_{t\searrow a} f(t)> \E(\mu)\geq f(a)$, up to taking $\delta$ small enough we have that $\psi_\mu(w)\geq \psi_\mu(\bar x) + J$ for some $J>0$. Let us also write $\psi_\mu=\psi_1+\psi_2$, where
\begin{align*}
\psi_1(x) = \int_{y \in B(\bar x,\eps)} \bar g(y-x)\, d\mu(y)\,, &&
\psi_2(x) = \int_{y \in \R^N\setminus B(\bar x,\eps)} \bar g(y-x) \,d\mu(y)\,.
\end{align*}
Since the function $\psi_2$ is clearly continuous in a small neighborhood of $\bar x$, up to further decreasing $\delta$ we must have
\begin{equation}\label{wmh}
\psi_1(w) \geq \psi_1(\bar x) + \frac J2\,.
\end{equation}
And finally, we find the contradiction since as already noticed $g$ must be strictly decreasing in $(0,2\eps)$, and since by construction $\mu$-a.e. $y\in B(\bar x,\eps)$ satisfies $|y-w|\geq |y-\bar x|$ then
\[
\psi_1(\bar x) = \int_{B(\bar x,\eps)} \bar g(y-\bar x)\,d\mu(y)
= \int_{B(\bar x,\eps)} g(|y-\bar x|)\,d\mu(y)
\geq \int_{B(\bar x,\eps)} g(|y-w|)\,d\mu(y) = \psi_1(w)\,,
\]
against~(\ref{wmh}).

\step{IV}{Proof if $f'>\eps$ in $(b-\eps,b)$.}
We are left with the last possible case to exclude, namely, that $f'>\eps$ in $(b-\eps,b)$. Our argument will be similar to the one of Step~III, we just need this time a little more care to deal with the geometry.\par

As in the previous case, we have a jump discontinuity at $b$, since $f(b)\leq \E(\mu)$ by lower semicontinuity of $\psi_\mu$ and Proposition~\ref{lowerlevel}, while $J:=\lim_{t\nearrow b} f(t)-\E(\mu)>0$ by Proposition~\ref{lowerlevel} and by assumption. Let now $\ell\ll b-a$ be a positive quantity, to be specified in a moment. This time, we write $\psi_\mu=\psi_1+\psi_2$ with
\begin{align*}
\psi_1(x) = \int_{y \in B(0,b+\ell)\cap B(x,2C_N\ell)} \bar g(y-x)\, d\mu(y) \,, &&
\psi_2(x) = \psi_\mu(x) - \psi_1(x)\,,
\end{align*}
where $C_N$ is the constant of Lemma~\ref{elgele}. The value of $\ell$ is so small that
\begin{align}\label{ssell}
\ell < \frac b{2C_N^2}\,, && \psi_1(\bar x) < \frac J 6\,.
\end{align}
It is again clear by construction that $\psi_2$ is continuous in a neighborhood of $\bar x=b{\rm e}_1$. As a consequence, up to further decreasing $\eps\ll \ell$ and calling this time $w=(b-\eps){\rm e}_1$, we have again~(\ref{wmh}). We claim now that, for any $b\leq r \leq b+\ell$, we have
\begin{equation}\label{sfat}
\int_{y \in \partial B(0,r)\cap B(\bar x,2C_N\ell)} \bar g(y-\bar x)\, d\H^{N-1}(y) \geq \frac 14\,\int_{y \in \partial B(0,r)\cap B(w,2C_N\ell)} \bar g(y-w)\, d\H^{N-1}(y)\,.
\end{equation}
Since $\mu$ is radial, by integration this will give $\psi_1(w)\leq 4\psi_1(\bar x)$, and this provides us with the searched contradiction thanks to~(\ref{ssell}) and~(\ref{wmh}). Therefore, to conclude we only have to establish~(\ref{sfat}).\par

Let us then fix $b\leq r\leq b+\ell$, and let us call $\xi=r-b+\eps\leq \ell + \eps < 2\ell$, which is the distance between $w$ and $\partial B(0,r)$. It is immediate to observe that, since $\ell\ll 1$, for any $y\in\spt\mu$ which belongs to the ball $B(\bar x,3C_N\ell)$ (which contains $B(w,2C_N\ell)$ since $\eps\ll \ell$), the implication
\[
|y-w| \leq C_N \xi \quad\Longrightarrow \quad |y-\bar x|\leq |y-w|
\]
holds. As a consequence, for any such $y$ we have $\bar g(y-\bar x)\geq \bar g(y-w)$ --indeed, as before we have that $\bar g$ is a radially strictly decreasing in a neighborhood of the origin, because otherwise $\psi_\mu$ could not be discontinuous. We deduce
\[\begin{split}
\int_{y \in \partial B(0,r)\cap B(w,C_N\xi)} \bar g(y-w)\, d\H^{N-1}(y) &\leq
\int_{y \in \partial B(0,r)\cap B(w,C_N\xi)} \bar g(y-\bar x)\, d\H^{N-1}(y)\\
&\leq \int_{y \in \partial B(0,r)\cap B(\bar x,C_N\xi)} \bar g(y-\bar x)\, d\H^{N-1}(y)\,.
\end{split}\]
Consequently, to conclude the validity of~(\ref{sfat}) we can limit ourselves to show
\[\begin{split}
\int_{\partial B(0,r)\cap B(\bar x,2C_N\ell)\setminus B(\bar x, C_N\xi)} \hspace{-20pt}&\hspace{20pt}\bar g(y-\bar x)\, d\H^{N-1}(y) \\
&\geq \frac 14\,\int_{\partial B(0,r)\cap B(w,2C_N\ell)\setminus B(w, C_N\xi)} \bar g(y-w)\, d\H^{N-1}(y)\,.
\end{split}\]
And in turn, this is clearly true if for any $C_N\xi<d<2C_N\ell$ and for any $\delta\ll \eps$ we have
\begin{equation}\label{conclcaraz}
\H^{N-1} \Big(\partial B(0,r) \cap B(\bar x,d+\delta)\setminus B(\bar x, d)\Big)
\geq \frac 14\,\H^{N-1} \Big(\partial B(0,r) \cap B(w,d+\delta)\setminus B(w, d)\Big)\,.
\end{equation}
Finally, this last inequality is a consequence of Lemma~\ref{elgele}. Indeed, take any $C_N\xi<d<2C_N\ell$, and call $\eta'=\xi$ and $\eta''=r-b$. By construction, $d > C_N \eta'> C_N \eta''$, and $r>C_N d$ by~(\ref{ssell}). Therefore, for any $\delta\ll 1$ we can apply Lemma~\ref{elgele} with constants $r,\, d,\, \eta',\, \delta$ as well as with constants $r,\,d,\,\eta'',\,\delta$, and then~(\ref{claimelgele}) gives~(\ref{conclcaraz}). As noticed before, this establishes~(\ref{sfat}) and then the proof is concluded.
\end{proof}

We conclude this section by observing an important consequence of the convexity of the support of an optimal measure, that is, the potential is continuous.

\begin{lemma}[Continuity of $\psi_\mu$]\label{psimucont}
Let $\bar g:\R^N\to\R$ be a radial, l.s.c., $L^1_{\rm loc}$ function such that, calling $g(|v|)=\bar g(v)$, the fuction $g$ is continuous in $(0,+\infty)$ and decreasing in a right neighborhood of $0$. Let also $\mu$ be an optimal measure, either in $\P$ or in $\P_{\rm rad}$, with support bounded and convex. Then the function $\psi_\mu$ is continuous. More precisely, there exists a continuous function $\tilde\psi:\R^N\to\R$ such that the set $\{\psi_\mu\neq \tilde\psi\}$ is negligible with respect to both $\mu$ and $\L^N$.
\end{lemma}
\begin{proof}
Calling for brevity $\Gamma=\spt(\mu)$, we simply define $\tilde \psi$ as the function which equals $\psi_\mu$ on $\R^N\setminus\Gamma$ and $\E(\mu)$ on $\Gamma$. The set where $\tilde\psi\neq \psi_\mu$ is the set of the points in $\Gamma$ where $\psi_\mu\neq \E(\mu)$, and this set is both $\mu$- and $\L^N$-negligible by Proposition~\ref{lowerlevel} and Remark~\ref{forevery}. As a consequence, we only have to show that $\tilde\psi$ is continuous.\par

Since in the open set $\R^N\setminus\Gamma$ we have that $\tilde\psi=\psi_\mu$ is continuous by construction, all we have to do is to show the continuity of $\tilde\psi$ at points of $\partial\Gamma$. Let us call $0<r<R$ two constants such that $g$ is decreasing in $(0,2r)$ and the diameter of $\Gamma$ is less than $R-r$, and let $\omega$ be the modulus of continuity of $g$ in the closed interval $[r,R]$. Let $y\notin \spt(\mu)$ be any point with ${\rm dist}(y,\Gamma)<r$, and let $x\in\Gamma$ be the point which minimizes the distance from $y$. We claim that
\begin{equation}\label{psicont}
\psi_\mu(y)-\psi_\mu(x) \leq \omega(|y-x|)\,,
\end{equation}
which will clearly conclude the thesis. By minimality of $x$, for every $z\in\Gamma$ we have $|z-y|\geq |z-x|$, thus $g(|z-y|)\leq g(|z-x|)$ if $|z-y|\leq 2r$. If, instead, $z\in \Gamma$ but $|z-y|>2r$, then we have also $|z-x|>r$, and then $g(|z-y|)-g(|z-x|) \leq \omega(|y-x|)$. As a consequence,
\[\begin{split}
\psi_\mu(y)-\psi_\mu(x) &= \int_{B(y,2r)} g(|z-y|)-g(|z-x|)\, d\mu(z) + \int_{\R^N\setminus B(y,2r)} g(|z-y|)-g(|z-x|)\, d\mu(z)\\
&\leq \omega(|y-x|) \mu\big(\R^N\setminus B(y,2r)\big)\leq \omega(|y-x|)\,,
\end{split}\]
which proves~(\ref{psicont}) and thus the thesis.
\end{proof}

\subsection{The main geometric estimates\label{ssect:main}}

This section is devoted to show three geometric estimates, that we will use to get the proof of Theorem~\mref{Linfty}. We start with an elementary calculation.
\begin{lemma}\label{Newteclem}
Let $\bar g$ be a function satisfying condition \Hz, let $r>0$ be given by \Hz, and let $\tilde r\leq r$. There exists $c=c(g,N,\tilde r)>0$ such that, defining $\tilde f:\R^N\to\R^+$ as $\tilde f(x) =1$ if $|x|<\tilde r$ and $\tilde f(x)=0$ otherwise, one has $\psi_{\tilde f}(z)\leq \psi_{\tilde f}(0) - c|z|^2$ for every $z\in\R^N$ with $|z| \ll 1$, depending on $g$ and $N$. The constant $c$ actually depends only on $N,\, \tilde r$ and $g'(\tilde r)$.
\end{lemma}
\begin{proof}
Let $z\in\R^N$ be a point with $\eta=|z|$ sufficiently small. For every $w\in \partial B(0,\tilde r)$, call $\Gamma(w)$ the segment joining $w$ and $w+z$, and $\theta=\theta(w)\in\S^1$ the angle between $w$ and $z$, that is, $w\cdot z = |w| |z| \cos\theta = \tilde r\eta \cos\theta$. We can then evaluate
\[\begin{split}
\psi_{\tilde f}(z) - \psi_{\tilde f}(0) &= \int_{B(z,\tilde r)} \bar g(y)\, dy- \int_{B(0,\tilde r)} \bar g(y)\, dy\\
&= \int_{w\in \partial B(0,\tilde r)} \int_{x\in\Gamma(w)} \bar g(x) \,d\H^1(x) \cos\theta\,d\H^{N-1}(w)\\
&= \int_{w\in \partial B(0,\tilde r)} \int_{t=0}^\eta \Big(g(\tilde r) + t\cos\theta g'(\tilde r) + o(\eta)\Big) \,dt \cos\theta\,d\H^{N-1}(w)\\
&= \frac{\eta^2}2 \, g'(\tilde r) \int_{w\in \partial B(0,\tilde r)} \cos^2\theta \,d\H^{N-1}(w)+ o(\eta^2)
= \frac{\eta^2}2 \, g'(\tilde r) C_N \tilde r^{N-1}+ o(\eta^2)\,.
\end{split}\]
Notice that $C_N$ is a purely geometrical constant, only depending on $N$. Its exact value, though elementary to calculate, is not important. Here, by $o(\eta)$ and $o(\eta^2)$ we denote a quantity which becomes arbitrarily smaller than $\eta$, or $\eta^2$, if $\eta$ is small enough, depending on $g,\, N$ and $\tilde r$. Since $g'(\tilde r)<0$, we obtain the thesis with $c=|g'(\tilde r)|C_N \tilde r^{N-1}/3$.
\end{proof}

We now pass to give an $L^\infty$ estimate in a very peculiar case. We will obtain the proof of Theorem~\mref{Linfty} basically reducing ourselves to this case.

\begin{lemma}[$L^\infty$ estimate]\label{mainlemma}
Let us assume that $\bar g$ satisfies \Hz, and let $f:\R^N\to\R$ be a positive, radial, ${\rm C}^2$ function, with unit $L^1$ norm, concentrated in $B_{\overline R}$ for some $\overline R>0$ and such that $\psi_f$ is constant in a neighborhood of $0$ and $f(0)=\max \{f(x)\}$. Then $f(0)\leq M_0$ for some constant $M_0=M_0(g,N,\overline R)$, which actually only depends on $N,\,\overline R$ and on the restriction of $g$ to $[r,\overline R]$.
\end{lemma}
\begin{proof}
First of all, we subdvide $\psi_f = \psi_1+\psi_2$, where
\begin{align*}
\psi_1(x) &= \int_{B_r} \bar g(y-x) f(y)\,dy \,, & \psi_2(x) = \int_{\R^N\setminus B_r} \bar g(y-x) f(y)\,dy \,.
\end{align*}
We start considering the function $\psi_2$, which is easier to deal with around $0$. In fact, of course $\psi_2$ is radial and of class ${\rm C}^2$ in $B_r$, so in particular $D\psi_2(0)=0$. Moreover, also keeping in mind that $\|f\|_{L^1}=1$, we have that
\begin{equation}\label{estipsi2}
\big| D^2 \psi_2(0)\big| = \bigg| \int_{\R^N\setminus B_r} D^2  \bar g(y-x) f(y)\,dy \bigg| \leq C(g,N,\overline R)\,.
\end{equation}
Notice that the constant $C$ actually only depends on $N$ and on $\max\limits_{t\in [r,\overline R]} \big\{|g'(t)|+|g''(t)|\big\}$.\par

We now pass to consider $\psi_1$. Notice that, for every $x\in B_r$, we have
\[\begin{split}
\psi_1(x)-\psi_1(0) &= \int_{B_r} \big(\bar g(y-x) - \bar g(y) \big) f(y)\,dy\\
&= \int_{B_r} \big(\bar g(y-x) - \bar g(y) \big) \big(f(y)-f(0)\big)\,dy + f(0) \int_{B_r} \left(\bar g(y-x) - \bar g(y)\right) \,dy\,.
\end{split}\]
The last integral in the above equation is nothing else than $\psi_{\tilde f}(x)-\psi_{\tilde f}(0)$ if we call $\tilde f$ the characteristic function of the ball $B_r$. We can then apply Lemma~\ref{Newteclem} and deduce from the above equation that
\begin{equation}\label{easyest}
\psi_1(x)-\psi_1(0) \leq -c f(0) |x|^2 + \int_{B_r} \big(\bar g(y-x) - \bar g(y) \big) \big(f(y)-f(0)\big)\,dy\,,
\end{equation}
where $c=c(g,N)$ is the constant given by Lemma~\ref{Newteclem} --notice that $c$ depends on $N$, $r$ and $g'(r)$, and in turn $r$ depends only on $g$ by assumption \Hz. We can now subdivide the last integral in two parts, namely, the integral in the smaller ball $B_{|x|}$, and the integral in $B_r\setminus B_{|x|}$. Since $f$ is radial and of class ${\rm C}^2$, we have that $|f(y)-f(0)|\leq \|D^2 f\|_{L^\infty} |x|^2$ for every $y\in B_{|x|}$, thus since $\bar g$ is integrable we deduce
\begin{equation}\label{smallball}\begin{split}
\bigg| \int_{B_{|x|}} &\big(\bar g(y-x) - \bar g(y) \big) \big(f(y)-f(0)\big)\,dy \bigg| \\
&\leq \|D^2 f\|_{L^\infty} |x|^2 \int_{B_{|x|}} \big|\bar g(y-x) - \bar g(y) \big|\, dy \leq 2 \|\bar g\|_{L^1(B_{2|x|})} \|D^2 f\|_{L^\infty} |x|^2\,.
\end{split}\end{equation}
We finally use that $\bar g$ is subharmonic in $B_{2r}\setminus \{0\}$. Indeed, this implies that for every $0<s<r$ the function
\[
z\mapsto \int_{\partial B_s} \bar g(z-y)\,d\H^{N-1}(y)
\]
is also subharmonic in $B_s$, and since this function is also radial by construction then its minimum is at $z=0$, that is,
\[
\int_{\partial B_s} \left(\bar g(z-y)-\bar g(y)\right)\,d\H^{N-1}(y) \geq 0\,.
\]
Using that fact that $0$ is a maximum point of the radial function $f$ in the ball $B_r$, by integration we immediately deduce that, again for every $x\in B_r$,
\[
\int_{B_r\setminus B_{|x|}} \big(\bar g(y-x)-\bar g(y)\big)\big(f(y)-f(0)\big)\,dy \leq 0\,.
\]
Putting this inequality together with~(\ref{smallball}) into~(\ref{easyest}), we obtain
\[
\psi_1(x)-\psi_1(0) \leq \Big(2 \|\bar g\|_{L^1(B_{2|x|})} \|D^2 f\|_{L^\infty}-c f(0)\Big) |x|^2\,.
\]
Since $\bar g\in L^1_{\rm loc}(\R^N)$, for $|x|$ small enough we deduce $\psi_1(x)-\psi_1(0) \leq -cf(0) |x|^2/2$. And finally, keeping in mind that $\psi_1+\psi_2=\psi_f$ is constant in a neighborhood of $0$ and~(\ref{estipsi2}), we deduce that $f(0)\leq M_0$ with
\[
M_0= 2 \,\frac{C(g,N,\overline R)}{Nc(g,N)}\,.
\]
The proof is then concluded. We underline that the constant $C(g,N,\overline R)$ only depends on $N$ and on the restriction of $g$ to $[r,\overline R]$, while $c(g,N)$ only depends on $N,\, r$ and $g'(r)$.
\end{proof}

\begin{remark}
Notice that the above estimate is true also if the origin is only a \emph{local} maximum of $f$. More precisely, if $0$ is a maximum of $f$ in the ball $B_{\hat r}$, then the above proof works substituting $r$ with $\tilde r:=\min\{r,\hat r\}$ (in fact, Lemma~\ref{Newteclem} is proved with $\tilde r$). Therefore, the $L^\infty$ bound in this more general case also depends on $\hat r$.
\end{remark}

We can now show an estimate on the possible mass of a small ball around the boundary of the support of an optimal measure.

\begin{lemma}[Estimate near the boundary]\label{nearbdr}
Let $\bar g$ be a function satisfying \Hz, and let $\overline\mu$ be an optimal measure (either in $\P$ or in $\P_{\rm rad}$) with support convex and contained in $B_{\overline R}$. There exists a constant $C=C(g,N,\overline R)$ such that, for every $x\in\partial\big({\spt}(\overline\mu)\big)$ and every $\rho\ll 1$, one has
\begin{equation}\label{massbndr}
\overline \mu \big( B(x,\rho) \big) \leq \frac C{|g'(2\rho)|}\,.
\end{equation}
\end{lemma}
\begin{proof}
Let us call for brevity $\Gamma=\spt(\overline\mu)$, and let $x\in \partial\Gamma$ be a given point. Since $\Gamma$ is convex, we can take an external direction $\nu\in\S^1$ to $\Gamma$ at $x$, that is, for every $y\in\Gamma$ one has $(y-x)\cdot\nu\leq 0$. We also write $\psi$ in place of $\psi_{\overline\mu}$ for the sake of simplicity of notations. We now take $\rho\ll 1$, and we observe that
\begin{equation}\label{stieaspsi}
\psi(x+2\rho\nu)-\psi(x)= \int_\Gamma \left(\bar g(x+2\rho\nu-y) - \bar g(y-x)\right)\,d\overline\mu(y)\,.
\end{equation}
The convexity of $\Gamma$ implies that, for every $y\in\Gamma$,
\begin{equation}\label{closer}
|x+2\rho\nu-y| \geq |y-x|\,.
\end{equation}
As a consequence,
\begin{equation}\label{forany}
\bar g(x+2\rho\nu-y)\leq \bar g(y-x) + \rho C\,,
\end{equation}
where $C=2\max \big\{ g'(t),\, 0<t<\overline R+1\big\}$. Notice carefully that $C=C(g,N,\overline R)$ is a well-defined real number thanks to the fact the we are maximizing $g'(t)$ instead of $|g'(t)|$, and this is possible thanks to~(\ref{closer}), which in turn is a consequence of the convexity of $\Gamma$.\par

While the estimate~(\ref{forany}) is true for every $y\in\Gamma$, let us now take $y\in \Gamma\cap B(x,\rho)$. For such a $y$, not only we have~(\ref{closer}), but we also have
\[
|x+2\rho\nu-y| - |y-x|\geq \rho\,,
\]
and then since $g$ is decreasing and $g'$ increasing near $0$ we have
\[\begin{split}
\bar g(x+2\rho\nu-y)&=g(|x+2\rho\nu-y|)\leq g(|y-x|+\rho) \leq g(|y-x|) +\rho g'(|y-x|+\rho)\\
&\leq g(|y-x|) +\rho g'(2\rho)=\bar g(y-x) -\rho \big|g'(2\rho)\big| \,.
\end{split}\]
Insterting in~(\ref{stieaspsi}) this estimate for $y\in \Gamma\cap B(x,\rho)$, and the estimate~(\ref{forany}) for points $y\in\Gamma \setminus B(x,\rho)$, we obtain
\[
\frac{\psi(x+2\rho\nu)-\psi(x)}\rho \leq -\big|g'(2\rho)\big| \overline\mu\big( B(x,\rho)\big) + C\,.
\]
By Proposition~\ref{lowerlevel}, also keeping in mind Remark~\ref{forevery}, we know that $\psi(x)\leq \E(\overline\mu)\leq \psi(x+2\rho\nu)$, and then the above inequality implies~(\ref{massbndr}), hence concluding the thesis.
\end{proof}

\subsection{The proof Theorem~\mref{Linfty}\label{ssect:pf}}

This section is devoted to present the proof of Theorem~\mref{Linfty}. We start with a first case.

\begin{lemma}[$L^\infty$ bound for a rapidly exploding $\bar g$]\label{grapidinc}
Let $\bar g$ satisfy \Hz, with $\lim_{t\to\infty} g(t)=+\infty$ and
\begin{equation}\tag{\ref{addass2}}
\limsup_{t\searrow 0} |g'(t)| t^N>0\,.
\end{equation}
Let moreover $\overline\mu$ be a minimizer of the energy, either in $\P$ or in $\P_{\rm rad}$, with convex support. Then $\overline\mu\in L^\infty$, and in particular $\|\overline\mu\|_{L^\infty}\leq M(g,N)$.
\end{lemma}
\begin{proof}
First of all, we apply Theorem~\mref{existence}, obtaining a constant $R=R(g,N)$ such that $\Gamma=\spt(\overline \mu)$ is contained in a ball of radius $R$. We set then $\overline R=2R+1$, we let $M_0(g,N,\overline R)$ and $C=C(g,N,\overline R)$ be the constants given by Lemma~\ref{mainlemma} and Lemma~\ref{nearbdr} respectively, and we define
\begin{equation}\label{defM}
M = \max \bigg\{ M_0 , \frac{2\cdot 4^N C}{\omega_N \alpha} \bigg\}\,,
\end{equation}
where $\alpha:=\limsup_{t\searrow 0} |g'(t)| t^N$. Notice that in most situations $\limsup_{t\searrow 0} |g'(t)| t^N =+\infty$, and in this case $M=M_0$. Notice also that $M$ only depends on $g$ and $N$, since $\overline R=\overline R(g,N)$.\par
For every $\rho\ll 1$, we consider a standard mollifier $\varphi_\rho:\R^N\to\R^+$, that is, a smooth, radial function supported in $B_\rho$ such that $\|\varphi_\rho\|_{L^1}=1$ and
\begin{equation}\label{goodmollif}
\|\varphi_\rho\|_{L^\infty} \leq \frac 2{\omega_N \rho^N}\,.
\end{equation}
We let then $\mu_\rho =\overline\mu\ast \varphi_\rho$, which is a positive, smooth function supported in a ball of radius $R+\rho$, and we claim that
\begin{equation}\label{mainest}
\| \mu_\rho \|_{L^\infty} \leq \max \bigg\{ M_0 , \frac{2}{\omega_N \rho^N}\cdot\sup_{x\in\partial\Gamma}\overline \mu\big(B(x,2\rho)\big) \bigg\}\,.
\end{equation}
We can easily show that this estimate concludes the proof. Indeed, we can take a sequence $\rho_j\searrow 0$ such that $g'(4\rho_j)(4\rho_j)^N\to \alpha$. Since the corresponding functions $\mu_{\rho_j}$ weakly* converge, in the sense of measures, to $\overline \mu$, we obtain the thesis, i.e. $\|\overline\mu\|_{L^\infty}\leq M$, directly by~(\ref{mainest}), since also thanks to Lemma~\ref{nearbdr} we have
\[
\lim_{j\to\infty} \frac{2}{\omega_N \rho_j^N}\, \overline \mu\big(B(x,2\rho_j)\big)
\leq \lim_{j\to\infty} \frac{2C }{\omega_N g'(4\rho_j) \rho_j^N}=\frac{2\cdot 4^N C }{\omega_N \alpha}\,.
\]
To conclude the thesis, we are then reduced to show the validity of~(\ref{mainest}). Let $y$ be a maximum point for the smooth function $\mu_\rho$. Suppose first that $y$ is contained in a $\rho$-neighborhood of $\partial\Gamma$, thus there exists some $x\in \partial\Gamma\cap B(y,\rho)$. In this case, also by~(\ref{goodmollif}) we have
\[
\|\mu_\rho\|_{L^\infty} = \mu_\rho(y) = \int_{B(y,\rho)} \varphi_\rho(y-z) d\overline\mu(z)
\leq \frac 2{\omega_N \rho^N} \, \overline\mu\big(B(y,\rho)\big)
\leq \frac 2{\omega_N \rho^N} \, \overline\mu\big(B(x,2\rho)\big)\,,
\]
hence~(\ref{mainest}) is established.\par

Let us now assume that the distance between $y$ and $\partial\Gamma$ is strictly greater than $\rho$, say $\rho+d$. In this case, the ball $B(y,\rho+d)$ is entirely contained either in $\Gamma$, or in $\R^N\setminus\Gamma$. However, this second case is impossible because it would imply $\mu_\rho(y)=0$, against the fact that $y$ is a maximum point for $\mu_\rho$, so we deduce $B(y,\rho+d)\subseteq \Gamma$. Now, we observe that $\psi_{\mu_\rho} = \psi_{\overline \mu}\ast \varphi_\rho$. Since $\psi_{\overline\mu}(z)=\E(\overline\mu)$ for a.e. $z\in\Gamma$ by Proposition~\ref{lowerlevel}, we deduce that $\psi_{\mu_\rho}(z)=\E(\overline\mu)$ for every $z \in B(y,d)$. Finally, we define $f$ as the radial average of $\mu_\rho$ around $y$, that is,
\[
f(x) = \intmed_{\partial B(y,|x|)} \mu_\rho(w)\, d\H^{N-1}(w)\,.
\]
Notice that $f$ is a smooth, radial function, with unit $L^1$ norm, supported in the ball $B_{2(R+\rho)}\subseteq B_{\overline R}$, and $0$ is a maximum point for $f$. Moreover, $\psi_f$ is constantly equal to $\E(\overline\mu)$ in the ball $B_d$. As a consequence, we can apply Lemma~\ref{mainlemma}, obtaining that $f(0)\leq M_0$, and since by construction $f(0)=\mu_\rho(y) = \|\mu_\rho\|_{L^\infty}$, we have obtained~(\ref{mainest}) also in this case and the proof is concluded.
\end{proof}

\begin{remark}
Notice that in the above lemma the assumption that $\lim_{t\to\infty} g(t)=+\infty$ has been used only to be able to apply Theorem~\mref{existence}, and in turn this was needed only to be sure that the support of $\overline\mu$ was contained in some ball. As a consequence, if $g$ does not explode at $\infty$ but a minimizer $\overline\mu$ has support which is convex and bounded, then it is still true that $\overline\mu$ is in $L^\infty$ (and in this case, the $L^\infty$ bound also depends on the diameter of the support).
\end{remark}

We want now to extend the $L^\infty$ bound in order to cover also cases when~(\ref{addass2}) does not hold. To do so, we will perturb the function $\bar g$ so to satisfy~(\ref{addass2}) and we will use the above lemma. It is simple to notice that the argument only works if we can approximate any measure with smooth functions in such a way that the energy converges. Therefore, we first have to show the following result. 

\begin{lemma}\label{lemapprox}
Assume that $g:(0,+\infty)\to \R^+$ is a continuous function, decreasing in a right neighborhood of $0$ and such that $\lim_{t\searrow 0} g(t) t^\alpha =0$ for some $0<\alpha<N$, and let $\bar g:\R^N\setminus \{0\}\to \R$ be given by $\bar g(x)=g(|x|)$. Then, for any probability measure $\mu\in\P$ with compact support there exists a sequence of smooth measures $\mu_j\in\P\cap {\rm C}^\infty_0(\R^N)$ weakly* converging to $\mu$ and such that
\begin{equation}\label{thesis315}
\lim_{j\to\infty} \E ( \mu_j) = \E(\mu)\,.
\end{equation}
Moreover, each measure $\mu_j$ belongs to $\P_{\rm rad}$ if so does $\mu$.
\end{lemma}
\begin{proof}
Let $g,\,\bar g$ and $\mu$ be as in the claim. As in the proof of Lemma~\ref{grapidinc}, for any $\rho>0$ we denote by $\varphi_\rho:\R^N\to\R$ a smooth, radial function supported in $B_\rho$, with unit $L^1$ norm and such that~(\ref{goodmollif}) holds. By lower semicontinuity of $\bar g$, for any sequence $\mu_j$ which weakly* converges to $\mu$ one has $\E(\mu)\leq \liminf \E(\mu_j)$. As a consequence, we can limit ourselves to consider the case when $\E(\mu)<+\infty$, since otherwise the claim is trivial, for instance one can define $\mu_j = \mu\ast \varphi_{1/j}$.\par

We start assuming that for some small, fixed $\eps>0$ one has
\begin{equation}\label{enlargeps}
\iint \bar g\bigg(\frac{y-x}{1+\eps}\bigg)\,d\mu(y)\,d\mu(x) <+\infty\,.
\end{equation}
Let us call $h(t)=g(t) t^\alpha$, which is a positive and continuous function which goes to $0$ when $t\searrow 0$. For any $j\in\N$, let us call $\rho= \min \big\{ t>0:\, h(3t/\eps)=1/j\}$. Keeping in mind that $\eps>0$ is small but fixed, we have that $\rho=\rho(j)>0$, and that $\lim_{j\to\infty} \rho=0$. We claim that
\begin{align}\label{goodgrho}
g(\rho)< \frac{3^\alpha g(3\rho/\eps)}{\eps^\alpha}\,, && g(3\rho/\eps) > \frac{\eps^\alpha(N-\alpha)}{3^\alpha 2^{N-\alpha}  N}\, \iintmed_{B_\rho\times B_\rho} \bar g(y-x)\,dy\,dx\,.
\end{align}
Indeed, the left inequality simply comes by observing that
\[
g(\rho) = \frac{h(\rho)}{\rho^{\alpha}} < \frac 1{j \rho^\alpha}=\frac{h(3\rho/\eps)}{\rho^\alpha} = \frac{3^\alpha g(3\rho/\eps)}{\eps^\alpha}\,.
\]
Concerning the right inequality, for any $0< t<2\rho$ one has $g(t) < 1/j t^{\alpha}$, thus
\[\begin{split}
\iint_{B_\rho\times B_\rho} \bar g(y-x)\,dy\,dx &\leq \omega_N \rho^N \int_{B_{2\rho}} \bar g(w)\,dw
= N\omega_N ^2\rho^N \int_{t=0}^{2\rho} g(t) t^{N-1}\, dt\\
&< \frac{N\omega_N ^2\rho^N }j \int_{t=0}^{2\rho} t^{N-1-\alpha}\, dt
= \frac{2^{N-\alpha}  N\omega_N^2\rho^{2N-\alpha}}{j(N-\alpha)}\\
&= \frac{3^\alpha 2^{N-\alpha}  N\omega_N^2\rho^{2N}}{\eps^\alpha(N-\alpha)}\,  g(3\rho/\eps) \,,
\end{split}\]
where in the last inequality we have used that $g$ is decreasing in $(0,3\rho/\eps)$, that is true by assumption as soon as $\rho\ll 1$, which in turn is true for $j$ large enough. We have then proved also the right inequality in~(\ref{goodgrho}).\par

We can now define $\mu_j=\mu\ast \varphi_\rho$, which is by construction a smooth probability measure with compact support, and which is radial if so is $\mu$. We can start calculating the energy of $\mu_j$ as
\begin{equation}\label{domconv}\begin{split}
\E(\mu_j)&=\iint \bar g(y-x)\, d\mu_j(y)\,d\mu_j(x)\\
&=\iint \bigg(\iint\bar g(y'-x')\, \varphi_\rho(y-y')   \varphi_\rho(x-x')\,dy' \,dx'\bigg)\, d\mu(y)  \, d\mu(x)\,.
\end{split}\end{equation}
Let us denote for brevity $\xi_j:\R^N\times \R^N\to \R^+$ as
\[
\xi_j(x,y) = \iint\bar g(y'-x')\, \varphi_\rho(y-y')   \varphi_\rho(x-x')\,dy' \,dx'\,,
\]
and let $r>0$ be such that $g$ is decreasing in $(0,r)$. If $j$ is large enough, we can assume that $6\rho<\eps r$. Let us then estimate $\xi_j(x,y)$ in three possible cases. First of all, assume that $|y-x|< 3\rho/\eps$. Then, by Riesz inequality, (\ref{goodmollif}) and both inequalities in~(\ref{goodgrho}) we have
\[\begin{split}
\xi_j(x,y) &\leq \iint\bar g(y'-x')\, \varphi_\rho(y')   \varphi_\rho(x')\,dy' \,dx' 
\leq \frac 4{\omega_N^2 \rho^{2N}} \iint_{B_\rho\times B_\rho}\bar g(y'-x') \,dy' \,dx'\\
&\leq \frac{3^\alpha 2^{N+2-\alpha}  N}{\eps^\alpha(N-\alpha)}\,  g(3\rho/\eps)
\leq \frac{3^\alpha 2^{N+2-\alpha}  N}{\eps^\alpha(N-\alpha)}\, \bar g(y-x)\,.
\end{split}\]
Second, assume that $3\rho/\eps\leq |y-x|<r/2$. In this case, we simply have
\[
\xi_j(x,y) \leq g(|y-x|-2\rho) \leq \bar g\bigg( \frac{y-x}{1+\eps}\bigg)\,.
\]
Finally, assume that $r/2\leq |y-x|\leq 2R$, where $R$ is a constant such that the support of $\mu$ is contained in a ball of radius $R$. In this case, we have
\[
\xi_j(x,y) \leq \max \big\{ g(t),\, |y-x|-2\rho\leq t\leq |y-x|+2\rho\big\} \leq 2 \bar g(y-x)\,,
\]
where the last inequality is true by the continuity of $g$ as soon as $\rho$ is small enough, hence again for any $j$ large enough. Putting together the last three estimates, we derive the existence of a constant $C=C(N,\alpha,\eps)$ such that for any $x,\,y \in (\spt\mu)^2$
\[
\xi_j(x,y) \leq C \bigg( \bar g(y-x) + \bar g \bigg( \frac{y-x}{1+\eps}\bigg)\bigg)\,.
\]
Since we are assuming that $\E(\mu)<+\infty$, as well as~(\ref{enlargeps}), the right hand side of the above inequality is integrable with respect to $\mu\otimes\mu$. Since the sequence $\xi_j(x,y)$ pointwise converges to $\bar g(y-x)$ when $j\to+\infty$, by the Dominated Converge Theorem and~(\ref{domconv}) we deduce that $\E(\mu_j)\to \E(\mu)$. In other words, $\{\mu_j\}$ is a sequence of smooth probability measures, radial if so is $\mu$, which weakly* converge to $\mu$ and which satisfy~(\ref{thesis315}). The proof is then concluded under the additional assumption~(\ref{enlargeps}).\par

Let us now assume that $\rho$ is a generic probability measure, not necessarily satisfying~(\ref{enlargeps}) for some $\eps>0$. For every $\eps>0$, let us now call $\tau_\eps:\R^N\to\R^N$ the function $\tau_\eps(x)= (1+\eps)x$, and let us set $\mu^\eps=(\tau_\eps)_\# \mu$, that is, for every function $\eta\in C_b(\R^N)$ we set
\[
\int \eta(x)\, d\mu^\eps(x) =\int \eta\big((1+\eps)x\big)\, d\mu(x)\,.
\]
Notice that also $\mu^\eps$ is a probability measure, radial if so is $\mu$, and that the sequence $\{\mu^\eps\}$ weakly* converge to $\mu$ when $\eps\searrow 0$. Moreover, since $g$ is decreasing in a right neighborhood of $0$ and $\mu$ has compact support, then again by the Dominated Convergence Theorem we have that
\[
\E(\mu^\eps) = \iint \bar g(y-x)\,d\mu^\eps(y)\,d\mu^\eps(x) = \iint \bar g\big((1+\eps)(y-x)\big)\,d\mu(y)\,d\mu(x) \freccia{\eps\searrow 0} \E(\mu)\,.
\]
In addition, $\mu^\eps$ satisfies~(\ref{enlargeps}) since
\[
\iint \bar g\bigg(\frac{y-x}{1+\eps}\bigg)\,d\mu^\eps(y)\,d\mu^\eps(x) = \iint \bar g(y-x)\,d\mu(y)\,d\mu(x) <+\infty\,.
\]
As a consequence, for every $\eps>0$ we can find a sequence $\{\mu^\eps_j\}$ of smooth probability measures, radial if so is $\mu$, which weakly* converge to $\mu^\eps$ and so that $\E(\mu^\eps_j)\to \E(\mu^\eps)$. The thesis follows then by a standard triangular argument.
\end{proof}

Thanks to the above approximation result, we are able to show the following $L^\infty$ bound without assuming the validity of~(\ref{addass2}).

\begin{lemma}[$L^\infty$ bound for a more general $\bar g$]\label{case3}
Let us assume that $\bar g$ satisfies \Hz\,and is subharmonic on $\R^N\setminus\{0\}$, as well that $\lim_{t\to\infty} g(t)=+\infty$. Then there exist a constant $M=M(g,N)$ and a measure $\bar\mu\in L^\infty$ which minimizes $\E$ in $\P_{\rm rad}$, with $\|\bar\mu\|_{L^\infty}\leq M$. If $N=1$, then there exists also a measure $\hat \mu\in L^\infty$ which minimizes $\E$ in the whole class $\P$, again with $\|\hat\mu\|_{L^\infty}\leq M$.
\end{lemma}
\begin{proof}
If $\bar g$ is strictly subharmonic in $B_r\setminus \{0\}$ and~(\ref{addass2}) holds, then \emph{every} minimal measure in $\P_{\rm rad}$, as well as \emph{every} minimal measure in $\P$ if $N=1$ (and such measures exist because the assumptions allow us to apply Theorem~\mref{existence}) has bounded and convex support by Proposition~\ref{convexN=1} or Proposition~\ref{convexN>1}, and then it satisfies the $L^\infty$ bound by Lemma~\ref{grapidinc}. We have then only to consider the case when~(\ref{addass2}) does not hold, or $\bar g$ is not \emph{strictly} subharmonic on some $B_r\setminus\{0\}$.\par

Let us first suppose that~(\ref{addass2}) does not hold (the case when~(\ref{addass2}) holds and $\bar g$ is not strictly subharmonic on some $B_r\setminus\{0\}$ is much simpler, and it will be discussed at the end of the proof). We can define $h:\R^+\to\R^+$ as the function such that $h(t)=0$ for $t\geq r$, while for $0<t<r$
\[
h(t) = t^{-N+\frac 12} +\frac{-4N^2-8N-3}8 \,r^{-N+\frac 12} + \frac{4N^2+4N-3}4\, r^{-N-\frac 12} t - \frac{4N^2-1}8\, r^{-N-\frac 32} t^2\,.
\]
An elementary calculation ensures that the function $\bar h:\R^N\setminus \{0\}\to\R^+$ given by $\bar h(x)=h(|x|)$ is ${\rm C}^2$, radial, subharmonic, strictly subharmonic in $B_r\setminus\{0\}$, and that $\limsup_{t\searrow 0} |h'(t)|t^N=+\infty$. We define then $g_\eps=g+\eps h$, and consistently $\bar g_\eps=\bar g+\eps\bar h$. Notice that $\bar g_\eps$ clearly satisfies assumption \Hz, is \emph{strictly} subharmonic in $B_r\setminus \{0\}$, and satisfies~(\ref{addass2}). For brevity of notations, we write $\E_h$ to denote the energy corresponding to the function $h$, and $\E_\eps$ to denote the energy corresponding to the function $g_\eps$, so in particular $\E_\eps=\E+\eps \E_h$.\par

We can apply Theorem~\mref{existence} with the function $\bar g_\eps$ in place of $\bar g$, and this ensures the existence of a minimizer of the energy $\E_\eps$ both in $\P$ and in $\P_{\rm rad}$. Notice that, calling $B$ the ball centered at the origin and with unit volume, $\E_\eps(\Chi{B})$ is a continuous function of $\eps$, converging to $\E(\Chi{B})\in (0,+\infty)$ when $\eps\searrow 0$. As a consequence, the claim of Theorem~\mref{existence} ensures the existence of some $R>0$, depending on $g$ and $N$ but not on $\eps$, such that if $\eps$ is small enough then the support of every measure minimizing the energy $\E_\eps$ in $\P$ or in $\P_{\rm rad}$ is contained in the ball $B(0,R)$ of radius $R$.\par

Let us then call $\mu_\eps$ a minimizer of $\E_\eps$ in $\P_{\rm rad}$. If $N=1$, we can also call $\mu_\eps$ a minimizer in $\P$. Every $\mu_\eps$ has support in the ball $B(0,R)$, and it has convex support. Indeed, if $N>1$ (hence $\mu_\eps$ minimizes the energy $\E_\eps$ in $\P_{\rm rad}$) then Proposition~\ref{convexN>1} ensures the support to be a closed ball. Instead, if $N=1$, then Proposition~\ref{convexN=1} ensures that the support is a closed segment, both if $\mu_\eps$ minimizes in $\P$ and in $\P_{\rm rad}$. We can the apply Lemma~\ref{grapidinc} to $\bar g_\eps$, obtaining that $\mu_\eps\in L^\infty$ and $\|\mu_\eps\|_{L^\infty}\leq M(g_\eps,\, N)=M(g,\eps,N)$.\par

We want to show that $M$ actually does not depend on $\eps$. To do so, we recall that by~(\ref{defM}) the value of $M$ coincides with $M_0$, since $\limsup_{t\searrow 0} |g_\eps'(t)|  t^N=+\infty$. And in turn, by Lemma~\ref{mainlemma}, the value of $M_0$ depends on $N,\, \overline R$ and on the restriction of $g_\eps$ to $[r,\overline R]$, where $\overline R=2R+1$. Since all the functions $g_\eps$ coincide with $g$ in $[r,\overline R]$, we have shown that $M$ only depends on $g$ and $N$, and not on $\eps$.\par

We can then find a sequence $\eps_j\searrow 0$ such that the measures $\mu_{\eps_j}$ weakly* converge to some $\bar\mu$, which is a probability measure since the measures $\mu_\eps$ are all concentrated in a same ball $B(0,R)$, and which is radially symmetric if so are all the measures $\mu_\eps$. By construction, we have that $\bar\mu$ is actually in $L^\infty$, and that $\|\bar\mu\|_{L^\infty}\leq M$. The proof will then be concluded once we show that $\bar\mu$ is a minimizer for the energy $\E$ in the class $\P_{\rm rad}$, or in the class $\P$ if $N=1$ and we are considering the minimization in $\P$.\par

To do so, let $\mu$ be any other probability measure, in $\P$ or in $\P_{\rm rad}$, and let us try to show that $\E(\bar\mu)\leq \E(\mu)$. First of all, we observe that the fact that~(\ref{addass2}) does not hold for $g$  means that $\lim_{t\searrow 0} g'(t)t^N=0$, and this immediately implies that $\lim_{t\searrow 0} g(t)t^{N-1}=0$. As a consequence, Lemma~\ref{lemapprox} gives us a sequence $\{\mu^n\}$ of smooth probability measures which weakly* converge to $\mu$ and so that $\E(\mu^n)\to \E(\mu)$. For any fixed $n\in\N$, by lower semicontinuity of the cost, the fact that $g\leq g_{\eps_j}$, and that $\mu_{\eps_j}$ minimizes the energy $\E_{\eps_j}$ in its class, we have
\begin{equation}\label{repeat}\begin{split}
\E(\bar\mu) &\leq \liminf_{j\to \infty} \E(\mu_{\eps_j})
\leq \liminf_{j\to \infty} \E_{\eps_j}(\mu_{\eps_j})\\
&\leq \liminf_{j\to\infty} \E_{\eps_j}(\mu^n)=\E(\mu^n) + \liminf_{j\to\infty}  \eps_j\E_h(\mu^n)
=\E(\mu^n)\,.
\end{split}\end{equation}
It is important to notice that the last equality holds because $\eps_j\to 0$ and because $\E_h(\mu^n)<+\infty$. And this last fact is true because $\mu^n$ is a smooth function with compact support, and $h\in L^1_{\rm loc}$. Directly using the measure $\mu$ in place of $\mu^n$ in the above chain of inequalities would clearly work if $\E_h(\mu)<+\infty$, but there is in general no guarantee that this is true, and this is why we had to use the regular measures $\mu^n$. Having observed that $\E(\bar\mu)\leq \E(\mu^n)$ for any generic $n\in\N$, and keeping in mind that $\E(\mu^n)\to\E(\mu)$ when $n\to\infty$, we conclude that $\E(\bar\mu)\leq \E(\mu)$ are required, and this ends the proof under the assumption that~(\ref{addass2}) does not hold.\par

To conclude the proof, we have then only to consider the case when~(\ref{addass2}) holds and $g$ fails to be strictly subharmonic on some $B_r\setminus \{0\}$ (the case when~(\ref{addass2}) holds and $g$ is strictly subharmonic on some $B_r\setminus \{0\}$ has been considered at the beginning). This case is much simpler than the one already studied. Indeed, this time it is enough to define $g_\eps = g + \eps t^2$ and to argue as before. Everything works without difficulties except for the fact that this time it is not necessarily true that $g(t) t^{N-1}\to 0$, so we are not allowed to use Lemma~\ref{lemapprox} to obtain the sequence $\{\mu^n\}$. However, there is no need to do so; indeed, as observed above, the only reason to use the sequence $\{\mu^n\}$ was that its regularity guaranteed that $\E_h(\mu^n)<+\infty$, while in general $\E_h(\mu)$ could have been $+\infty$. But this time, the function $h$ is simply given by $h(t)=t^2$, hence the fact that $\E_h(\mu)<+\infty$ is surely true because $\mu$ is compactly supported, and then one can directly use $\mu$ in place of $\mu^n$ in~(\ref{repeat}) without using Lemma~\ref{lemapprox}. The proof is then finished.
\end{proof}

The proof of Theorem~\mref{Linfty} is then concluded. Indeed, case~(1) is considered in Lemma~\ref{grapidinc}; case~(2) follows from case~(1) since the assumptions guarantee that the optimal measure have convex support by Proposition~\ref{convexN=1} and Proposition~\ref{convexN>1}; and case~(3) is considered in Lemma~\ref{case3}.

\begin{remark}\label{ConjFrLi}
We conclude this section with a comment about the importance of Theorem~\mref{Linfty} in connection with the different possible settings of the problem. Let us recall that, as described in the Introduction, the minimization of the energy has been considered in the literature among sets of given measure, or among functions with values in $[0,1]$ and given $L^1$ norm, or among positive measures with given mass. Each setting clearly extends the previous ones, so in particular if a minimizer within a class belongs to one of the previous classes, then it is also a minimizer for that class.\par
A simple, yet deep difference between the problem among measures and the problem among sets, or functions with values in $[0,1]$, is the dependance of the minimizers with the mass. Indeed, the homogeneity of the energy $\E$ among measures implies that dependance on the mass of the minimal measures is trivial, in the sense that the minimizing measures for different masses are simply a multiple of each other. On the contrary, when considering the problem among sets, or among functions with values in $[0,1]$, the different scaling properties of the attractive and repulsive part of the energy imply that the solutions heavily depend on the mass. For instance, Frank and Lieb in~\cite{FL} show that for some values of the parameters in~(\ref{prototype}) the solutions among functions with values in $[0,1]$ are actually characteristic functions (of balls) whenever the mass is large enough. However, this is no more true for smaller mass, and in general one could expect solutions which are functions having also intermediate values between $0$ and $1$. Even more, Frank and Lieb conjecture in~\cite[Remark~2 (3)]{FL} that for some values of the parameters solutions are never characteristic functions, whatever the mass is.\par

Let us now notice that, if Theorem~\mref{Linfty} applies, then the minimizing probability measure is an $L^\infty$ function. As a consequence, the minimizer among the measures with mass $m$ is a function with $L^\infty$ norm smaller than $1$ as soon as $m$ is small enough. And then, it is also a minimizer for the problem among $L^1$ functions with values in $[0,1]$. This implies that, under the assumptions of Theorem~\mref{Linfty}, the minimizers of the problem among functions with values in $[0,1]$ are surely not characteristic functions of sets if the mass is small enough.
\end{remark}

\section{On the uniqueness and radiality of the optimal measure\label{sect:unique}}

This last section is devoted to discuss the question of the uniqueness (of course, up to translations) and radiality of the optimal measures, and in particular to prove Theorem~\mref{Unique}. We start noticing that, for a general function $\bar g$, there is no reason why there should be a unique optimal measure, and it is easy to build examples in which the uniqueness fails. Less clear is the question whether optimal measures should be radial. Indeed, on one side, the problem is rotationally invariant; but on the other side, in most of the examples the function $g$ has a unique minimum point, say $t_{\rm min}$, and then, roughly speaking, points would like to stay at distance $t_{\rm min}$ from each other, and this somehow pushes against the radiality. The two questions are related, because obviously if there is a unique optimal measure then it must be radial, since all the rotations of this measure are also optimal and then have to coincide with it.

As already pointed out in the Introduction, the question of the uniqueness has already been treated in some papers. More precisely, in~\cite{BCT} it is shown that there is a unique optimal measure whenever $\bar g(x)=|x|^{-\alpha} + |x|^2$ with any $0<\alpha<N$. Later on, the same result was obtained in~\cite{Lop} for $\bar g(x)=|x|^{-\alpha}+|x|^\beta$ with any $2\leq \beta\leq 4$, and again $0<\alpha<N$. We can now further generalise the result by substituting the term $|x|^{-\alpha}$ with any positive definite function, basically using the same approach of~\cite{Lop} and the properties of the positive definite functions. We start with the definition.

\begin{defin}[Strongly positive definite functions]\label{defposdef}
Given a $L^1_{\rm loc}$ function $\bar h:\R^N\to\R^+$, for any two measures $\mu,\,\nu\in\P(\R^N)$ we define
\[
\E_{\bar h}(\mu,\nu) = \iint_{\R^N\times\R^N} \bar h(y-x)\, d\mu(x)\,d\nu(y)\,.
\]
Then, we say that \emph{$\bar h$ is strongly positive definite} if for any $\mu,\,\nu\in\P(\R^N)$ one has
\begin{equation}\label{tocheck}
\E_{\bar h}(\mu,\mu)+\E_{\bar h}(\nu,\nu) \geq 2 \E_{\bar h}(\mu,\nu)  \,,
\end{equation}
and that \emph{$\bar h$ is strictly strongly positive definite} if the inequality is strict whenever $\mu\neq \nu$.
\end{defin}
The definition of positive definite functions is quite standard, see for instance~\cite{Rudin,LL,M}, and it simply consists in asking the validity of the property~(\ref{tocheck}) when $\mu$ and $\nu$ are $L^1$ functions with unit $L^1$ norm, or characteristic functions of sets of unit volume --these two choices are equivalent by a simple approximation argument. We add the word ``strongly'' to remember that our assumption is in principle stronger, since we want to test with every probability measure. However, it is simple to see that the two notions are in fact equivalent, at least for functions for which the measures can be approximated by functions in energy, as we show now.

\begin{lemma}\label{equivstr}
Assume that $\bar h:\R^N\to\R^+$ is a $L^1_{\rm loc}$ function with the property that for any measure $\mu\in\P(\R^N)$ there is a sequence of smooth functions $\mu_j\in \P\cap {\rm C}^\infty_0(\R^N)$ weakly* converging to $\mu$ and such that $\E_{\bar h}(\mu_j,\mu_j) \to \E_{\bar h}(\mu,\mu)$ for $j\to\infty$. Then, $\bar h$ is strongly positive definite if and only if it is positive definite.
\end{lemma}
\begin{proof}
Of course, whenever $\bar h$ is strongly positive definite, then it is also positive definite, so we only have to prove the opposite implication. Let us then assume that $\bar h$ is positive definite, and let $\mu$ and $\nu$ be two probability measures. By assumption, we can take two sequences $\{\mu_j\}$ and $\{\nu_n\}$ in $\P\cap {\rm C}^\infty_0(\R^N)$ which weakly* converge to $\mu$ and $\nu$ respectively, and such that
\begin{align*}
\lim_{j\to \infty} \E_{\bar h}(\mu_j,\mu_j) = \E_{\bar h}(\mu,\mu)\,, && \lim_{n\to \infty} \E_{\bar h}(\nu_n,\nu_n) = \E_{\bar h}(\nu,\nu)\,.
\end{align*}
Since~(\ref{tocheck}) is true for any pair $(\mu_j,\, \nu_n)$, by lower semicontinuity we have
\[\begin{split}
2\E_{\bar h}(\mu,\nu) &\leq \liminf_{j\to\infty} 2\E_{\bar h}(\mu_j,\nu)
\leq \liminf_{j\to\infty} \liminf_{n\to\infty} 2\E_{\bar h}(\mu_j,\nu_n)\\
&\leq \liminf_{j\to\infty} \liminf_{n\to\infty} \E_{\bar h}(\mu_j,\mu_j) + \E_{\bar h}(\nu_n,\nu_n)
=\E_{\bar h}(\mu,\mu)+\E_{\bar h}(\nu,\nu)\,.
\end{split}\]
Therefore, $\bar h$ is strongly positive definite.
\end{proof}

As said above, the notion of positive definiteness is well known. In particular, the following sufficient conditions for the positive definiteness are known, see for instance~\cite{NP}.
\begin{theorem}\label{knowndefpos}
Let $\bar h:\R^N\to \R^+$ be a $L^1_{\rm loc}$ function. If $\lim_{|x|\to \infty} \bar h(x) = \inf \bar h$ and $\bar h$ is subharmonic in $\R^N\setminus\{0\}$, then $\bar h$ is positive definite. If $\limsup_{|x|\to\infty} \bar h(x)<+\infty$ and the Fourier transform of $\bar h$ is a positive Borel measure, then $\bar h$ is positive definite.
\end{theorem}

In particular, the function $t^{-\alpha}$ is strictly positive definite for every $0<\alpha<N$, as well as the Gaussian function $e^{-|x|^2/2}$, and the function $-\ln(|x|)$ in dimension $N=2$.

An important corollary of Lemma~\ref{equivstr}, which immediately follows by Lemma~\ref{lemapprox}, is that a radial function $\bar h$ with the property that $h$, defined by $\bar h(x)=h(|x|)$, is continuous, decreasing in a right neighborhood of $0$, and so that $\lim_{t\searrow 0} h(t) t^\alpha=0$ for some $0<\alpha<N$, is strongly positive definite as soon as it is positive definite. If $\bar h$ is subharmonic in $\R^N\setminus\{0\}$, there is not even need of asking the existence of some $\alpha$ as above. More precisely, the following result holds.

\begin{lemma}\label{subh->spd}
Every radial function $\bar h:\R^N\to\R^+$ which is radially decreasing and subharmonic in $\R^N\setminus \{0\}$ is strongly positive definite.
\end{lemma}
\begin{proof}
Let $\bar h:\R^N\to\R^+$ be a radial, radially decreasing and subharmonic function, and let us call as usual $h:\R^+\to\R^+$ the function given by $\bar h(x)=h(|x|)$. By Theorem~\ref{knowndefpos} we know that $\bar h$ is positive definite. Therefore, as noticed above, Lemma~\ref{equivstr} and Lemma~\ref{lemapprox} readily give that $\bar h$ is strongly positive definite if there is some $0<\alpha<N$ such that $\lim_{t\searrow 0} h(t) t^\alpha=0$.\par

Let us then suppose then that it is not so. For any small $\eps>0$, we call $\bar h_\eps:\R^N\setminus\{0\}\to\R^+$ the radial function which coincides with $\bar h$ in $\R^N\setminus B_\eps$ and which is harmonic in $B_\eps\setminus\{0\}$. In particular, the corresponding function $h_\eps$ is defined in $(0,\eps)$ by the formula
\[
h_\eps(t) =\left\{\begin{array}{cc}
\bal \frac{\big(t^{2-N}-\eps^{2-N}\big)h'(\eps)}{(2-N)\eps^{1-N}}+h(\eps) \eal& \hbox{if $N\neq 2$}\,, \\[10pt]
\eps\big(\ln(t)-\ln(\eps)\big)h'(\eps)+h(\eps) & \hbox{if $N=2$}\,.
\end{array}\right.
\]
Notice that $\bar h_\eps$ satisfies by construction all the requirements of the lemma, and moreover $\lim_{t\searrow 0} h_\eps(t) t^{N-1/2}=0$. As a consequence, $\bar h_\eps$ is strongly positive definite. Given then any two probability measures $\mu$ and $\nu$, we know that
\[
\E_\eps(\mu,\mu)+\E_\eps(\nu,\nu) \geq 2 \E_\eps(\mu,\nu)  \,,
\]
where we write $\E_\eps$ in place of $\E_{\bar h_\eps}$ for simplicity of notations. Since the functions $\bar h_\eps$ pointwise converge to $\bar h$ when $\eps\to 0$, and $\bar h_\eps\leq \bar h$ by construction and since $\bar h$ is subharmonic, by the Dominated Convergence Theorem each term in the above inequality converges for $\eps\to 0$ to the corresponding term with $\E_{\bar h}$ in place of $\E_\eps$. The validity of~(\ref{tocheck}) is then established, and so $\bar h$ is strongly positive definite.
\end{proof}

The importance in this context of the notion of positive definiteness is mainly given by the following elementary observation.

\begin{lemma}\label{posconv}
Let $\bar g:\R^N\to\R^+$ be a l.s.c., $L^1_{\rm loc}$ and strongly positive definite function. Then the energy $\E$ is convex in the space
\[
\dot{\P}_{\bar g,c} = \bigg\{ \mu \in \P(\R^N):\, \spt(\mu)\comp\R^N,\, \E(\mu)<+\infty,\, \int x \, d\mu=0 \bigg\}
\]
of the probability measure with compact support, finite energy, and baricenter in the origin. If $\bar g$ is strictly strongly positive definite then $\E$ is strictly convex in $\dot{\P}_{\bar g,c}$.
\end{lemma}
\begin{proof}
Let $\mu,\,\nu$ be two measures in $\dot{\P}_{\bar g,c}$, and let $0\leq \lambda\leq 1$. Then,
\[\begin{split}
\E\big(\lambda \mu + (1-\lambda)\nu\big) &=
\E_{\bar g}\big(\lambda \mu + (1-\lambda)\nu,\lambda \mu + (1-\lambda)\nu\big)\\
&= \lambda^2 \E(\mu) + (1-\lambda)^2 \E(\nu) + 2 \lambda(1-\lambda) \E_{\bar g}(\mu,\nu)\,,
\end{split}\]
and then by~(\ref{tocheck})
\[
\E\big(\lambda \mu + (1-\lambda)\nu\big) - \big(\lambda\E(\mu)+(1-\lambda)\E(\nu)\big)
=\lambda(\lambda-1) \Big(\E_{\bar g}(\mu,\mu) +\E_{\bar g}(\nu,\nu) - 2 \E_{\bar g}(\mu,\nu)\Big)\leq 0\,,
\]
which gives the required convexity of $\E$. If $\bar g$ is strictly strongly positive definite, then the above inequality is strict whenever $0<\lambda<1$ and $\mu\neq\nu$, thus $\E$ is strictly convex in $\dot{\P}_{\bar g,c}$.
\end{proof}

We can now conclude the paper by giving the proof of Theorem~\mref{Unique}.

\proofof{Theorem~\mref{Unique}}
The function $x\mapsto |x|^\beta$ is positive definite for $2\leq \beta\leq 4$, as proved in~\cite[Theorem~2.1]{Lop}, and then it is also strongly positive definite since Lemma~\ref{equivstr} can be clearly applied. As a consequence, if $\bar h$ is strongly positive definite, so is also $\bar g(x)=\bar h(x)+|x|^\beta$. Since the assumptions allow us to apply Theorem~\mref{existence}, we know the existence of an optimal measure $\mu$ for the energy $\E$, which is compactly supported. Up to a translation, we can assume that $\mu$ has baricenter in the origin. For every $\theta\in\S^{N-1}$, we call then $\mu_\theta$ the measure obtained by rotating $\mu$ of an angle $\theta$. By radiality of $\bar g$, each measure $\mu_\theta$ has the same energy as $\mu$, so they are all optimal. Since the energy $\E$ is convex in the space $\dot{\P}_{\bar g,c}$, which contains by construction all the measures $\mu_\theta$, we deduce the optimality also of the measure
\[
\bar\mu = \intmed_{\S^{N-1}} \mu_\theta\, d\H^{N-1}(\theta)\,,
\]
which is of course in $\P_{\rm rad}$. The existence of an optimal measure in $\P_{\rm rad}$ is then established.\par

If $\bar h$ is strictly strongly positive definite, then so is $\bar g$ and then the energy $\E$ is strictly convex in $\dot{\P}_{\bar g,c}$. As a consequence, all the measures $\mu_\theta$ have to coincide, and this means that $\mu=\bar \mu$ is radial. We claim that $\mu$ is actually the only optimal measure with baricenter in the origin. This is in fact obvious, since if there is another such optimal measure $\nu\neq \mu$, the strict convexity of $\E$ gives that $\E\big((\mu+\nu)/2\big)< \E(\mu)$, which is absurd.\par

Let us now assume that $\bar h$ is subharmonic in $\R^N\setminus \{0\}$ and radially decreasing. The fact that $\bar h$ is strongly positive definite is given by Lemma~\ref{subh->spd}, and then the existence of a minimizer $\bar\mu$ of $\E$ in $\P$ which is radial follows by the first part of the proof. Suppose now that $\bar h$ is strictly subharmonic in some $B_r\setminus\{0\}$. Since of course $\bar\mu$ in particular minimizes the energy in $\P_{\rm rad}$, the fact that the support of $\bar\mu$ is a ball is given by either Proposition~\ref{convexN=1} or Proposition~\ref{convexN>1}. Let us call $R_1$ the radius of this ball. To concude the thesis, we have to show that also the support of any other minimal measure in $\P$ (which might also not belong to $\P_{\rm rad}$) is a ball of radius $R_1$.\par

To do so, let us consider two optimal measures $\mu',\, \mu''$ in $\dot{\P}_{\bar g,c}$, and let us call $S',\, S''$ their supports. Keep in mind that by Proposition~\ref{lowerlevel} we have $\psi_{\mu'}\equiv \E(\bar\mu)$ in $S'$ and $\psi_{\mu'}\geq \E(\bar\mu)$ in $\R^N\setminus S'$, and the same is true with $\psi_{\mu''}$ and $S''$ in place of $\psi_{\mu'}$ and $S'$. By the convexity of $\E$ in $\dot{\P}_{\bar g,c}$, we deduce that also $\mu=(\mu'+\mu'')/2$ is an optimal measure, so $\psi_\mu\leq \E(\bar\mu)$ everywhere on the support of $\mu$, which is $S'\cup S''$. Since $\psi_\mu = (\psi_{\mu'}+\psi_{\mu''})/2$, also keeping in mind Remark~\ref{forevery} we deduce that actually $\psi_{\mu'}(x)=\E(\bar\mu)$ for every point $x\in S''\setminus S'$. However, the function $\psi_{\mu'}$ is subharmonic in $\R^N\setminus S'$, and strictly subharmonic in the subset of $\R^N\setminus S'$ made by the points with distance less than $r$ from $S'$. Since $\psi_{\mu'}$ is constant in $S''\setminus S'$, this means that all the points of $S''\setminus S'$ have distance at least $r$ from $S'$.\par

Let then $\mu$ be a generic optimal measure in $\P$, and suppose up to a translation that $\mu\in \dot{\P}_{\bar g,c}$. Applying the above considerations with $\mu'=\mu$ and $\mu''$ equal to some rotation $\mu_\theta$ of $\mu$, we immediately deduce that the support of $\mu$ must be radial. We can now observe that the support of $\mu$ is actually a ball. Indeed, since it is radial then it coincides with the support of the radial average
\[
\intmed_{\S^{N-1}} \mu_\theta\, d\H^{N-1}(\theta)\,,
\]
and since this latter measure is a minimal measure which belongs to $\P_{\rm rad}$, by Proposition~\ref{convexN=1} or~\ref{convexN>1} its support must be a ball. That is, we have observed that the support of a generic optimal measure $\mu$ in $\P$ must be a ball. And finally, applying again the above considerations with $\mu'=\mu$ and $\mu''=\bar\mu$, we deduce that the radius of this ball must coincide with the radius $R_1$ of the support of $\bar\mu$. The proof is then concluded.
\end{proof}

\end{document}